\documentclass{amsart}
\usepackage{amssymb}
\usepackage{amsmath}
\usepackage{amsthm}
\usepackage{mathrsfs}
\usepackage{colortbl}
\usepackage{graphicx}
\newtheorem{theorem}{Theorem}
\newtheorem{lemma}[theorem]{Lemma}

\newtheorem{remark}[theorem]{Remark}

\theoremstyle{definition}

\newtheorem{proposition}[theorem]{Proposition}

\makeatletter

\@addtoreset{theorem}{section}
\makeatother

\makeatletter

\@addtoreset{equation}{section}
\makeatother

\begin{document}                                                 

\title[Stability of the diffusion equation on unbounded surfaces]{Stability and maximal regularity of the diffusion equation on unbounded surfaces}                                
\author[Hajime Koba]{Hajime Koba}                                
\address{Faculty of Advanced Science and Technology, Kumamoto University, 2-39-1 Kurokami, Chuo-ku, Kumamoto, 860-8555, Japan}                            
%\curraddr{...}                                   
\email{koba-hajime@kumamoto-u.ac.jp}

%\date{}                                      
%\thanks{This work was partly supported by the Japan Society for the Promotion of Science (JSPS) KAKENHI Grant Number JP21K03326.}                                     
%\translator{...}
\keywords{Surface diffusion; Diffusion equation; Stability; Maximal regularity; Laplace-Beltrami operator}                       
\subjclass[]{35B35, 35A01, 35R01, 47D06, 53A05}

\begin{abstract}
This paper studies both stability and maximal regularity of the diffusion equation on an unbounded surface without any assumptions on the curvature. We show that the surface diffusion system admits a unique global-in-time strong solution satisfying $L^2$-asymptotic stability when the slope of the surface is gentle. We also derive maximal $L^2$-regularity of the solution when the exterior force satisfies both surface divergence form and tangential condition. Moreover, we investigate the H\"{o}lder continuity with respect to time of the solution. The key idea of showing our stability and maximal regularity results is to apply nice properties of both Laplace  and weighted Laplace-Beltrami operators in order to derive the properties of the Laplace-Beltrami operator on the unbounded surface.
\end{abstract}

\maketitle

\section{Introduction}\label{sect1}

We are interested in both stability and maximal regularity of the diffusion equation on an unbounded surface without any assumptions on the curvature. The system is one of the simple diffusion models for the Earth's surface. This paper has three purposes. The first one is to study the properties of the Laplace-Beltrami operator on an unbounded surface. The second one is to show the existence of a global-in-time strong solution to the diffusion system on the unbounded surface. The third one is to investigate $L^2$-asymptotic stability, maximal $L^2$-regularity, and the H\"{o}lder continuity of the solution.

Let us first introduce basic notation. We use the characters $i,j$ as $1,2,3$, and the characters $\alpha, \beta,\alpha', \beta'$ as $1,2$, that is, $i , j \in \{ 1,2,3 \}$ and $\alpha, \beta, \alpha',\beta' \in \{ 1,2 \}$. Let $x = { }^t (x_1,x_2,x_3) \in \mathbb{R}^3$, $X = { }^t (X_1,X_2) \in \mathbb{R}^2$ be the spatial variables, $t , \tau, s \geq 0$ be the time variables. The symbol $\nabla$ and $\nabla_X$ are two gradient operators defined by $\nabla = { }^t (\partial_1 , \partial_2, \partial_3)$ and $\nabla_X = { }^t (\partial_{X_1}, \partial_{X_2})$, where $\partial_j = \partial/{\partial x_j}$ and $\partial_{X_\alpha} = \partial/{\partial X_\alpha}$. Write $\Delta_X = \partial_{X_1}^2 + \partial_{X_2}^2$ and $\partial_t = \partial/{\partial t}$. For $1 \leq p \leq \infty$ and $m \in \mathbb{N}$, the symbols $L^p(\mathbb{R}^2)$ and $W^{m,p}(\mathbb{R}^2)$ denote the usual Lebesgue and Sobolev spaces, and $\Vert \cdot \Vert_{L^p(\mathbb{R}^2)}$ and $\Vert \cdot \Vert_{W^{m,p} ( \mathbb{R}^2)}$ denote its norms. Let $\left< \cdot , \cdot \right>_{L^2(\mathbb{R}^2)}$ be the $L^2$-inner product on $L^2(\mathbb{R}^2)$. The symbol $d \mathcal{H}_x^2$ denotes the $2$-dimensional Hausdorff measure. For a linear operator $\mathcal{A}$ on a Banach space, the symbols $\mathcal{R} (\mathcal{A})$, $\mathcal{N}(\mathcal{A})$, $\rho ( \mathcal{A})$, and ${\rm{e}}^{ t \mathcal{A} }$ denote the range of $\mathcal{A}$, the null set of $\mathcal{A}$, the resolvent set of $\mathcal{A}$, and the semigroup generated by $\mathcal{A}$.

This paper considers both stability and maximal regularity of strong solutions to the \emph{diffusion equation on an unbounded surface} $\Gamma$:
\begin{equation}\label{eq11}
\begin{cases}
\partial_t u - \kappa \Delta_\Gamma u = \mathcal{F} &\text{ on }\Gamma \times (0,\infty ),\\
u \vert_{t =0} = u_0 & \text{ on }\Gamma,\\
{\displaystyle{ \lim_{x \in \Gamma, { \ } \vert (x_1,x_2) \vert \to \infty} u = 0}} &\text{ on } (0,\infty ),
\end{cases}
\end{equation}
where the unknown function $u = u(x,t)$ is the \emph{concentration of amount of a substance} on $\Gamma$, while the given positive number $\kappa$ is the \emph{diffusion coefficient}, the given function $\mathcal{F} = \mathcal{F} (x,t)$ is the external force, $u_0 = u_0(x)$ is the given initial datum, and the surface $\Gamma$ is represented by
\begin{equation}\label{eq12}
\Gamma = \{ x \in \mathbb{R}^3;{ \ }x_1 = X_1, x_2=X_2,x_3 = \varphi (X_1,X_2), { \ }(X_1,X_2) \in \mathbb{R}^2 \}
\end{equation}
for some $\varphi \in BC^2 (\mathbb{R}^2) $. Here $\Delta_\Gamma$ denotes the \emph{Laplace-Beltrami operator} defined by
\begin{equation*}
\Delta_\Gamma \psi  = \sum_{j=1}^3 (\partial_j^\Gamma)^2 \psi { \ }\text{ and }{ \ }\partial_j^\Gamma \psi = \sum_{i=1}^3 (\delta_{ij} - n_j n_i) \partial_i \psi = \partial_j \psi - n_j (n \cdot \nabla ) \psi,
\end{equation*}
where $\delta_{ij}$ is the Kronecker delta, and $n= n(x) = { }^t (n_1,n_2,n_3)$ is the unit outer normal vector at $x \in \Gamma$. Note that $\Delta_\Gamma \psi = \nabla_\Gamma \cdot \nabla_\Gamma \psi$, where $\nabla_\Gamma = { }^t (\partial_1^\Gamma , \partial_2^\Gamma , \partial_3^\Gamma)$. See Section \ref{sect2} for differential operators $\Delta_\Gamma$ and $\partial^\Gamma_j$.

Let us state mathematical analysis of diffusion systems on a manifold or a fixed surface. Jimbo \cite{Jim84} considered the stable equilibrium solutions of the semilinear diffusion equation on a compact Riemannian manifold with a non-negative Ricci curvature. They showed that any non-constant equilibrium solution of their diffusion system is unstable. Davies \cite[Chapter 5]{Dav89} dealt with the heat kernel of a complete Riemannian manifold. Under the conditions that the Ricci curvature of the manifold is bounded below by a negative constant, they derived fundamental properties of the heat semigroup. Mazzucato-Nistor \cite{MN06} studied maximal $L^p$-$L^q(1<p,q<\infty)$ regularity of the heat kernel of noncompact manifolds. They showed that system \eqref{eq11} admits a unique local-in-time strong $L^q$-solution when $u_0$ is in a real interpolation space and $F$ satisfies that $F = {\rm{div}}_\Gamma G(u)$, where $G (\cdot)$ is a locally Lipschitz map. Shao-Simonett \cite{SS14} investigated the continuous maximal regularity of their elliptic operators on uniformly regular Riemannian manifolds without boundary. They applied their theory to derive the analyticity of the solutions to the diffusion system on their manifold. Bandle-Monticelli-Punzo \cite{BMP18} studied the stability of time-periodic solutions to their reaction-diffusion system on a manifold. This paper studies both stability and maximal regularity of the diffusion equation on an unbounded surface without any assumptions on the curvature. In \cite{Jim84} they considered the stability of solutions to the diffusion system on a compact manifold, while this paper investigates it on unbounded surfaces. In \cite{MN06} they studied a local-in-time maximal $L^p$-$L^q$ regularity of the heat kernel of noncompact manifolds, while this paper investigates a global-in-time maximal $L^2$-regularity of the heat semigroup of unbounded surfaces.

Recently, some researchers have been studying their diffusion systems on an evolving surface. Dziuk-Elliott \cite{DE07} used a Galerkin-type method to show the existence of a unique weak solution to the diffusion system on an evolving closed surface. Alphonse-Elliott-Stinner \cite{AES15a}, \cite{AES15b} introduced their evolving Hilbert space to construct a weak solution of several PDEs on evolving surfaces. Djurdjevac \cite{Dju17} and Djurdjevac-Elliott-Kornhuber-Ranner \cite{DEKR18} studied the diffusion system with random coefficients on an evolving closed surface. Djurdjevac \cite{Dju17} showed the existence of a unique mean-weak solution to their surface diffusion system. In \cite{DEKR18}, they investigated the stability property of solutions to their semi-discrete problem of their surface diffusion system. Koba \cite{Kob22} studied the existence of a global-in-time strong solution to their advection-diffusion equation on an evolving surface with a boundary.

We briefly state the main results of this paper; see Theorems \ref{thm25}-\ref{thm27} for details. 
\begin{theorem}\label{thm11}
Let $\kappa >0$ and $\varphi \in BC^2(\mathbb{R}^2)$. Then there is $\delta_* = \delta_* (\kappa) >0$ such that if
\begin{equation*}
\Vert \nabla_X \varphi \Vert_{L^\infty (\mathbb{R}^2)} \leq \delta_*
\end{equation*}
then for each $0< \eta <1$,
\begin{equation*}
u_0 \in W^{1,2} (\Gamma) \text{ and } \mathcal{F} \in C_{loc}^{\eta}((0,\infty);L^2(\Gamma)) \cap L^2(0,\infty ; L^2(\Gamma)),
\end{equation*}
system \eqref{eq11} admits a unique global-in-time strong solution $u$ in
\begin{equation*}
C([0,\infty); L^2 (\Gamma)) \cap C((0,\infty);W^{2,2}(\Gamma)) \cap C^1((0,\infty);L^2(\Gamma)).
\end{equation*}
Moreover, the solution $u$ satisfies the five properties:\\
\noindent $(\rm{i})$ $[\rm{Energy { \ }equality}]$ For all $0 \leq s < t< \infty$,
\begin{multline*}
\Vert u(t) \Vert_{L^2(\Gamma )}^2 + 2 \kappa \int_s^t \Vert \nabla_\Gamma u (\tau ) \Vert_{L^2(\Gamma)}^2 { \ }d \tau\\
 = \Vert u(s ) \Vert_{L^2(\Gamma )}^2 + \int_s^t \int_{\Gamma} \mathcal{F}(\tau) u(\tau)  { \ }d \mathcal{H}_x^2 d \tau .
\end{multline*}
\noindent $(\rm{ii})$ $[$\rm{H\"{o}lder continuity}$]$
\begin{equation*}
u, u_t, \Delta_\Gamma u \in C^{ \eta }_{loc}((0,\infty); L^2(\Gamma)).
\end{equation*}
\noindent $(\rm{iii})$ $[\rm{Maximal { \ } regularity}]$ There is $C_1 = C_1(\kappa, \Vert \nabla_X^2 \varphi \Vert_{L^\infty (\mathbb{R}^2)} )>0$ such that for each $T>0$
\begin{multline*}
\Vert u_t \Vert_{L^2(0, T ; L^2(\Gamma))} + \Vert \kappa \Delta_\Gamma u \Vert_{L^2(0,T ; L^2(\Gamma))}\\
 \leq C_1 \Vert u_0 \Vert_{W^{1,2} (\Gamma)} + C_1(1 +\sqrt{ T}) \Vert \mathcal{F} \Vert_{L^2(0,T ; L^2(\Gamma))} .
\end{multline*}
Assume in addition that there exists $\mathcal{Q}= { }^t (\mathcal{Q}_1,\mathcal{Q}_2,\mathcal{Q}_3) \in L^2(0, \infty ; W^{1,2}( \Gamma))$ such that 
\begin{equation*}
\mathcal{F} = \nabla_\Gamma \cdot \mathcal{Q} (= \partial_1^\Gamma \mathcal{Q}_1 + \partial_2^\Gamma \mathcal{Q}_2 + \partial_3^\Gamma \mathcal{Q}_3 ),
\end{equation*}
and that
\begin{equation*}
\mathcal{Q} \cdot n = 0.
\end{equation*}
Then there is $C_2 = C_2 ( \kappa, \Vert \nabla_X^2 \varphi \Vert_{L^\infty (\mathbb{R}^2)}  ) >0$ such that
\begin{multline*}
\Vert u_t \Vert_{L^2(0,\infty ; L^2(\Gamma))} + \Vert \kappa \Delta_\Gamma u \Vert_{L^2(0,\infty ; L^2(\Gamma))}\\
 \leq C_2 \Vert u_0 \Vert_{W^{1,2} (\Gamma)} + C_2 \Vert \mathcal{Q} \Vert_{L^2(0,\infty ; W^{1,2}(\Gamma))}.
\end{multline*}
\noindent $(\rm{iv})$ $[\rm{Functions{ \ } on { \ } surfaces}]$ Assume in addition that $ \mathcal{F} \in L^2(\Gamma \times (0,\infty))$. Then for each fixed $T \in (0,\infty)$
\begin{equation*}
u \in L^2(\Gamma \times (0,T)).
\end{equation*}
\noindent $(\rm{v})$ $[\rm{Stability}]$ Assume in addition that $ \mathcal{F} \equiv 0$. Then
\begin{equation*}
\lim_{t \to \infty} \Vert u(t) \Vert_{L^2(\Gamma)} = 0.
\end{equation*}
\end{theorem}
\noindent See Appendix for the function spaces $L^2(\Gamma \times (0,\infty))$ and $L^2(\Gamma \times (0,T))$.

 Applying our methods, we also obtain the following results.
\begin{proposition}\label{prop12}
Let $\kappa >0$, $\varphi \in BC^2(\mathbb{R}^2)$, and let $\delta_*$ and $C_1$ be the two positive constants appearing in Theorem \ref{thm11}. If $\Vert \nabla_X \varphi \Vert_{L^\infty (\mathbb{R}^2)} \leq \delta_*$, then for each $0 < \eta <1$, $T \in (0,\infty)$, $u_0 \in W^{1,2} (\Gamma)$, and 
\begin{equation*}
\mathcal{F} \in C_{loc}^{\eta}((0,T);L^2(\Gamma)) \cap L^2(0,T ; L^2(\Gamma)) \cap L^2( \Gamma \times (0,T) ),
\end{equation*}
system \eqref{eq11} admits a unique local-in-time strong solution $u$ in
\begin{equation*}
C([0,T); L^2 (\Gamma)) \cap C((0,T);W^{2,2}(\Gamma)) \cap C^1((0,T);L^2(\Gamma)) \cap L^2(\Gamma \times (0,T)),
\end{equation*}
satisfying $u, u_t, \Delta_\Gamma u \in C^{  \eta }_{loc}((0,T); L^2(\Gamma))$ and
\begin{multline*}
\Vert u_t \Vert_{L^2(0, T ; L^2(\Gamma))} + \Vert \kappa \Delta_\Gamma u \Vert_{L^2(0,T ; L^2(\Gamma))}\\
 \leq C_1 \Vert u_0 \Vert_{W^{1,2} (\Gamma)} + C_1(1 +\sqrt{ T}) \Vert \mathcal{F} \Vert_{L^2(0,T ; L^2(\Gamma))} .
\end{multline*}
\end{proposition}

Let us state some key ideas of this paper. Set $\widehat{x}(X_1,X_2) = { }^t (\widehat{x}_1,\widehat{x}_2,\widehat{x}_3) = { }^t (X_1,X_2, \varphi (X_1,X_2))$, $\widehat{u} = \widehat{u}(X,t) = u (\widehat{x}(X) ,t)$, $\widehat{u}_0 = \widehat{u}_0(X) = u_0 (\widehat{x}(X))$, and $\widehat{\mathcal{F}} = \widehat{\mathcal{F}} (X,t) =\mathcal{F}(\widehat{x} (X) ,t)$. Then we have
\begin{equation}\label{eq13}
\begin{cases}
\partial_t \widehat{u} + L \widehat{u} = \widehat{\mathcal{F}} &\text{ in }\mathbb{R}^2 \times (0, \infty),\\
\widehat{u} \vert_{t =0} = \widehat{u}_0 & \text{ in }\mathbb{R}^2,\\
{\displaystyle{\lim_{\vert X \vert \to \infty} \widehat{u} =0}} &\text{ on }(0,\infty).
\end{cases}
\end{equation}
Here $L$ is the \emph{Laplace-Beltrami operator} defined by
\begin{equation*}
L f = - \sum_{\beta=1}^2 \sum_{\alpha =1}^2\frac{\kappa}{\sqrt{G}} \frac{\partial}{\partial X_\alpha} \bigg( \sqrt{G} g^{\alpha \beta} \frac{\partial f }{\partial X_\beta} \bigg) .
\end{equation*}
See Section \ref{sect2} for the derivation of system \eqref{eq13} and the notations $G$, $g^{\alpha \beta}$. This paper studies both stability and maximal $L^2$-regularity of system \eqref{eq13}. However, $L$ is not a selfadjoint operator and system \eqref{eq13} does not have an energy structure. To overcome these difficulties, we apply nice properties of the Laplace operator $- \Delta_X$ and the \emph{weighted Laplace-Beltrami operator} defined by
\begin{equation*}
\mathcal{L} f = \sqrt{G} L f = - \sum_{\beta=1}^2 \sum_{\alpha =1}^2 \kappa \frac{\partial}{\partial X_\alpha} \bigg( \sqrt{G} g^{\alpha \beta} \frac{\partial f }{\partial X_\beta} \bigg)
\end{equation*}
to study the Laplace-Beltrami operator $L$.

The outline of this paper is as follows. In Section \ref{sect2}, we introduce our settings and state the main results of this paper. In Section \ref{sect3}, we prepare some tools such as the Laplace operator $- \Delta_X$ and surface divergence theorem. In Section \ref{sect4}, we study the Laplace-Beltrami operator $L$ by applying the Laplace operator $- \Delta_X$. In Section \ref{sect5}, we investigate maximal $L^2$-regularity of system \eqref{eq13}. In Section \ref{sect6}, we study basic properties of the weighted Laplace-Beltrami operator $\mathcal{L}$. In Section \ref{sect7}, we derive the asymptotic stability of solutions to our system by applying both maximal $L^2$-regularity of the Laplace-Beltrami operator $L$ and nice properties of the weight Laplace-Beltrami operator $\mathcal{L}$. In Appendix, we characterize our function spaces on surfaces, and introduce one useful tool to derive the H\"{o}lder continuity of solutions to our system.

In this paper, we make use of the semigroup theory, the differential geometry, and the maximal regularity theory to derive our results. We refer the reader to \cite{Paz83}, \cite{EN00} for the semigroup theory, \cite{Cia05}, \cite{Jos11} for the differential geometry, and \cite{DHP03}, \cite{KW04} for maximal $L^p$-regularity. From here we often use the Einstein summation convention. For example,
\begin{equation*}
g^{\alpha \beta}g_\alpha = \sum_{\alpha =1}^2 g^{\alpha \beta}g_\alpha \text{ or } g^{\alpha \beta}\frac{\partial^2 f}{\partial X_\alpha \partial X_\beta} = \sum_{\alpha,\beta =1}^2g^{\alpha \beta}\frac{\partial^2 f}{\partial X_\alpha \partial X_\beta}. 
\end{equation*}

%\newpage

\section{Settings and Main Results}\label{sect2}

We define notations and state the main results of this paper. Let $\Gamma \subset \mathbb{R}^3$ be a 2-dimensional surface, and $\varphi = \varphi (X_1,X_2) \in BC^2 (\mathbb{R}^2)$. Assume that the surface $\Gamma$ can be represented by \eqref{eq12}. From now on we fix $\varphi$.

Let us first define fundamental notations. For all $X_1,X_2 \in \mathbb{R}$, we set
\begin{equation*}\widehat{x} = \widehat{x}(X_1,X_2) = \begin{pmatrix} \widehat{x}_1\\ \widehat{x}_2\\ \widehat{x}_3
\end{pmatrix} =
\begin{pmatrix}
X_1 \\
X_2\\
\varphi (X_1,X_2)
\end{pmatrix}.
\end{equation*}
It is clear that
\begin{equation*}
\Gamma = \{ x \in \mathbb{R}^3;{ \ }x= \widehat{x}(X_1,X_2), { }^t(X_1,X_2) \in \mathbb{R}^2 \}.
\end{equation*}
For every $X = { }^t(X_1,X_2) \in \mathbb{R}^2$, we define
\begin{equation*}
g_1 = g_1 (X) = \frac{\partial \widehat{x}}{\partial X_1}  = \begin{pmatrix}
1\\
0\\
\partial_{ X_1 } \varphi
\end{pmatrix},{ \ }
g_2 = g_2 (X) = \frac{\partial \widehat{x}}{\partial X_2}  = \begin{pmatrix}
0\\
1\\
\partial_{ X_2 } \varphi
\end{pmatrix},
\end{equation*}
\begin{align*}
g_{11} & = g_{11}(X) = g_1 \cdot g_1 = 1 + (\partial_{X_1}\varphi)^2,\\
g_{22} & = g_{22}(X) = g_2 \cdot g_2 =1 + (\partial_{X_2}\varphi)^2 ,\\
g_{12} & = g_{12}(X) = g_1 \cdot g_2= (\partial_{X_1}\varphi) (\partial_{X_2}\varphi),\\
g_{21} & = g_{21}(X) = g_2 \cdot g_1 =(\partial_{X_1}\varphi) (\partial_{X_2}\varphi),\\
G & = G (X) = g_{11} g_{22} - g_{12} g_{21} = 1 + (\partial_{X_1}\varphi)^2 + (\partial_{X_2}\varphi)^2.
\end{align*}
It is clear that $G \geq 1$ for all $X \in \mathbb{R}^2$. Moreover, we set
\begin{align*}
\begin{pmatrix}
g^{11} & g^{12}\\
g^{21} & g^{22}
\end{pmatrix} = \begin{pmatrix}
g_{11} & g_{12}\\
g_{21} & g_{22}
\end{pmatrix}^{-1} = \frac{1}{g_{11}g_{22} - g_{12} g_{21}  }\begin{pmatrix}
g_{22} & -g_{12}\\
-g_{21} & g_{11}
\end{pmatrix}\\ = \frac{1}{1 + (\partial_{X_1}\varphi)^2 + (\partial_{X_2}\varphi)^2 } \begin{pmatrix}
1 +  (\partial_{X_2}\varphi)^2 &   -(\partial_{X_1}\varphi) (\partial_{X_2}\varphi)&\\
 - (\partial_{X_1}\varphi) (\partial_{X_2}\varphi)& 1 +(\partial_{X_1}\varphi)^2
\end{pmatrix}.
\end{align*}
It is easy to check that
\begin{equation*}
g_1 \times g_2 = \begin{pmatrix}
- \partial_{X_1} \varphi\\
- \partial_{X_2} \varphi\\
1
\end{pmatrix}, { \ }G = \vert g_1 \times g_2 \vert^2,
\end{equation*}
and that the unit outer normal vector $n = n(x) = { }^t (n_1,n_2,n_3)$ at $x \in \Gamma$ can be represented by
\begin{equation*}
n (\widehat{x}(X)) = \frac{g_1 \times g_2}{ \vert g_1 \times g_2 \vert} = \frac{1}{\sqrt{ 1 + (\partial_{X_1} \varphi )^2 + (\partial_{X_2} \varphi)^2 } } \begin{pmatrix}
- \partial_{X_1} \varphi\\
- \partial_{X_2} \varphi\\
1
\end{pmatrix}.
\end{equation*}

Direct calculations give the estimates.
\begin{lemma}\label{lem21}For each $X \in \mathbb{R}^2$ and $\alpha, \beta \in \{ 1,2 \}$,
\begin{align*}
& 1  \leq G (X) \leq 1 + 2 \Vert \nabla_X \varphi \Vert_{L^\infty(\mathbb{R}^2)}^2,\\  
& \frac{1}{1 + 2 \Vert \nabla_X \varphi \Vert_{L^\infty(\mathbb{R}^2)}^2} \leq \frac{1}{G(X)}  \leq 1,\\
& \bigg\Vert \frac{\partial G}{\partial X_\alpha} \bigg\Vert_{L^\infty (\mathbb{R}^2)} \leq 4 \Vert \nabla_X \varphi \Vert_{W^{1,\infty} (\mathbb{R}^2)}^2,\\
& \bigg\Vert \frac{\partial g^{\alpha \beta}}{\partial X_\alpha} \bigg\Vert_{L^\infty (\mathbb{R}^2)} \leq  6 \Vert \nabla_X^2 \varphi \Vert_{L^{\infty} (\mathbb{R}^2)}.
\end{align*}
\end{lemma}

Next, we recall differential operators on surfaces. Set $\nabla_\Gamma = { }^t (\partial_1^\Gamma , \partial_2^\Gamma , \partial_3^\Gamma)$ and $\Delta_\Gamma = (\partial_1^\Gamma)^2 + (\partial_2^\Gamma)^2 + (\partial_3^\Gamma)^2$, where $\partial_j^\Gamma \psi = \partial_j \psi -n_j (n \cdot \nabla ) \psi$. From \cite[Lemma 3.1]{Kob23} (see \cite{Cia05}, \cite{Jos11}), we have the following representation formulas.
\begin{lemma}\label{lem22}
$(\rm{i})$ For all $\psi,\psi_\sharp,\psi_\flat \in C^\infty (\mathbb{R}^3)$,
\begin{align*}
\int_\Gamma \psi { \ }d \mathcal{H}^2_x &= \int_{\mathbb{R}^2} \widehat{\psi} \sqrt{G} { \ }dX,\\
\int_\Gamma \partial^\Gamma_j \psi { \ }d \mathcal{H}^2_x &= \int_{\mathbb{R}^2} g^{\alpha \beta} \frac{\partial \widehat{x}_j }{\partial X_\alpha} \frac{ \partial \widehat{\psi} }{\partial X_\beta} \sqrt{G} { \ }dX,\\
\int_\Gamma \nabla_\Gamma \psi_\sharp \cdot \nabla_\Gamma \psi_\flat { \ }d \mathcal{H}^2_x & = \int_{\mathbb{R}^2} g^{\alpha \beta} \frac{\partial \widehat{\psi}_\sharp }{\partial X_\alpha} \frac{ \partial \widehat{\psi}_\flat }{\partial X_\beta} \sqrt{G} { \ }dX,\\
\int_\Gamma \Delta_\Gamma \psi { \ }d \mathcal{H}^2_x & = \int_{\mathbb{R}^2} \frac{1}{\sqrt{G}} \frac{\partial}{\partial X_\alpha} \bigg( \sqrt{G} g^{\alpha \beta} \frac{\partial \widehat{\psi}}{\partial X_\beta} \bigg)  \sqrt{G} { \ }dX,
\end{align*}
where $\widehat{\psi} = \widehat{\psi} (X) = \psi (\widehat{x}(X) )$.\\
$(\rm{ii})$ For all $\Psi = { }^t (\Psi_1, \Psi_2,\Psi_3) \in [C^\infty (\mathbb{R}^3)]^3$ and $\Phi \in C^\infty (\mathbb{R}^4)$,
\begin{align*}
\int_\Gamma \nabla_\Gamma  \cdot \Psi { \ }d \mathcal{H}^2_x &= \int_{\mathbb{R}^2} g^{\alpha \beta}g_\beta \cdot \frac{\partial \widehat{\Psi} }{\partial X_\alpha} \sqrt{G} { \ }dX,\\
\int_\Gamma \frac{\partial \Phi}{\partial t} { \ }d \mathcal{H}_x^2 &= \int_{\mathbb{R}^2} \frac{\partial \widehat{\Phi}}{\partial t} \sqrt{G} { \ }d X. 
\end{align*}
Here $\widehat{\Psi} = \widehat{\Psi} (X) = \Psi (\widehat{x}(X) )$ and $\widehat{\Phi} = \widehat{\Phi} (X,t) = \Phi (\widehat{x}(X),t)$. 
\end{lemma}

From Lemmas \ref{lem21} and \ref{lem22}, we have the lemma.
\begin{lemma}\label{lem23}
Let $\psi \in C^\infty (\mathbb{R}^3)$ and $i,j \in \{ 1,2,3 \}$. Set $\widehat{\psi} = \widehat{\psi}(X) = \psi (\widehat{x}(X))$. Assume that $\widehat{\psi} \in W^{2,2}(\mathbb{R}^2)$. Write
\begin{align*}
\Vert \psi \Vert_{L^2(\Gamma)} & = \bigg( \int_{\mathbb{R}^2} \vert \widehat{\psi} \vert^2 \sqrt{G} { \ }dX \bigg)^{1/2} = \Vert G^{1/4} \widehat{\psi} \Vert_{L^2( \mathbb{R}^2 )},\\
\Vert \partial_j^\Gamma \psi \Vert_{L^2(\Gamma)} & = \bigg( \int_{\mathbb{R}^2} \left\vert g^{\alpha \beta} \frac{\partial \widehat{x}_j }{\partial X_\alpha} \frac{ \partial \widehat{\psi} }{\partial X_\beta} \right\vert^2 \sqrt{G} { \ }dX \bigg)^{1/2},\\
\Vert \partial_i^\Gamma \partial_j^\Gamma \psi \Vert_{L^2(\Gamma)} & = \bigg( \int_{\mathbb{R}^2} \left\vert g^{\alpha' \beta'} \frac{\partial \widehat{x}_i }{\partial X_{\alpha'}} \frac{ \partial }{\partial X_{\beta'}} \bigg( g^{\alpha \beta} \frac{\partial \widehat{x}_j }{\partial X_\alpha} \frac{ \partial \widehat{\psi} }{\partial X_\beta} \bigg) \right\vert^2 \sqrt{G} { \ }dX \bigg)^{1/2}.
\end{align*}
Then there is $C = C(\Vert \nabla_X \varphi \Vert_{W^{1,\infty}(\mathbb{R}^2)}) >0$ such that
\begin{equation*}
\Vert \psi \Vert_{L^2(\Gamma)} + \Vert \partial_j^\Gamma \psi \Vert_{L^2(\Gamma)} +\Vert \partial_i^\Gamma \partial_j^\Gamma \psi \Vert_{L^2(\Gamma)}  \leq C \Vert \widehat{\psi} \Vert_{W^{2,2}(\mathbb{R}^2)}.
\end{equation*}
\end{lemma}

For ease of understanding, we write that for $f_0 \in L^2(\mathbb{R}^2)$ and $f \in W^{1,2}(\mathbb{R}^2)$,
\begin{align}
\Vert f_0 \Vert_{\mathscr{L}^2(\Gamma)} &:= \Vert G^{1/4} f_0 \Vert_{L^2(\mathbb{R}^2)},\label{eq21}\\
\Vert f \Vert_{\dot{\mathscr{W}}^{1,2} (\Gamma) } &:= \bigg( \int_{\mathbb{R}^2} g^{\alpha \beta} \frac{\partial f }{\partial X_\alpha} \frac{ \partial f }{\partial X_\beta} \sqrt{G} { \ }dX \bigg)^{1/2},\label{eq22}\\
\Vert f \Vert_{\mathscr{W}^{1,2} (\Gamma) } &:= ( \Vert f \Vert_{\mathscr{L}^2(\Gamma)}^2 + \Vert f \Vert_{\dot{\mathscr{W}}^{1,2}(\Gamma)}^2 )^{1/2}.\label{eq23}
\end{align}
From Lemmas \ref{lem21} and \ref{lem42}, we have the estimates.
\begin{lemma}\label{lem24}
If $\Vert \nabla_X \varphi \Vert_{L^\infty (\mathbb{R}^2)} \leq 1$, then for every $f_0 \in L^2(\mathbb{R}^2)$ and $f \in W^{1,2} (\mathbb{R}^2)$ 
\begin{align*}
\Vert f_0 \Vert_{L^2(\mathbb{R}^2)} & \leq \Vert f_0 \Vert_{\mathscr{L}^2(\Gamma)} \leq 3 \Vert f_0 \Vert_{L^2(\mathbb{R}^2)},\\
(1/3)\Vert \nabla_X f \Vert_{L^2(\mathbb{R}^2)} & \leq \Vert f \Vert_{\dot{\mathscr{W}}^{1,2}(\Gamma)} \leq 3 \Vert \nabla_X f \Vert_{L^2(\mathbb{R}^2)},\\
(1/3)\Vert f \Vert_{W^{1,2}(\mathbb{R}^2)} & \leq \Vert f \Vert_{\mathscr{W}^{1,2}( \Gamma )} \leq 3 \Vert f \Vert_{W^{1,2}(\mathbb{R}^2)}.
\end{align*}
\end{lemma}

Applying Lemma \ref{lem22} into system \eqref{eq11}, we have \eqref{eq13}. Based on Lemma \ref{lem23}, we consider the following abstract system:
\begin{equation}\label{eq24}
\begin{cases}
{\displaystyle{ \frac{dv}{dt} + Lv = F }} \text{ on }(0,\infty),\\
v\vert_{t =0} = v_0.
\end{cases}
\end{equation}
Here $v_0 \in L^2(\mathbb{R}^2)$, $F \in L^2(0,\infty ;L^2(\mathbb{R}^2))$, $\kappa >0$, and
\begin{equation}\label{eq25}
\begin{cases}
{\displaystyle{L f := - \frac{\kappa}{\sqrt{G}} \frac{\partial}{\partial X_\alpha} \bigg( \sqrt{G} g^{\alpha \beta} \frac{\partial f }{\partial X_\beta} \bigg),}}\\
D (L) := W^{2,2} (\mathbb{R}^2).
\end{cases}
\end{equation}
We also call $L$ the \emph{Laplace-Beltrami operator}.

We state the main results of this paper.
\begin{theorem}[Energy equality and H\"{o}lder continuity]\label{thm25} Let $\kappa >0$, $\varphi \in BC^2(\mathbb{R}^2)$, $v_0 \in L^2 (\mathbb{R}^2)$, and $F \in L^2(0,\infty;L^2(\mathbb{R}^2)) \cap C_{loc}^\eta ((0,\infty); L^2 (\mathbb{R}^2))$ for some $0< \eta <1$. Assume that 
\begin{equation*}
\Vert \nabla_X \varphi \Vert_{L^\infty (\mathbb{R}^2)} \leq \frac{1}{4} .
\end{equation*}Then system \eqref{eq24} admits a unique global-in-time strong solution $v$ in
\begin{equation*}
C([0,\infty) ; L^2(\mathbb{R}^2) ) \cap C((0,\infty) ; W^{2,2} (\mathbb{R}^2) ) \cap C^1((0,\infty) ; L^2 (\mathbb{R}^2)).
\end{equation*}
Moreover, the solution $v$ satisfies that for all $0 \leq s < t <+ \infty$
\begin{multline}\label{eq26}
\Vert G^{1/4} v(t) \Vert_{L^2(\mathbb{R}^2 )}^2 + 2 \kappa \int_s^t \int_{\mathbb{R}^2} g^{\alpha \beta} \frac{\partial v }{\partial X_\alpha} \frac{ \partial v }{\partial X_\beta} \sqrt{G} { \ }dX { \ }d \tau\\
 = \Vert G^{1/4} v(s ) \Vert_{L^2(\mathbb{R}^2 )}^2 + \int_s^t \int_{\mathbb{R}^2} F(\tau) v(\tau) \sqrt{G} { \ }d X d \tau,
\end{multline}
and that
\begin{equation*}
G^{1/4} v, { \ }G^{1/4}Lv,{ \ } G^{1/4}{d v}/{dt} \in  C_{loc}^{ \eta } ((0,\infty); L^2 (\mathbb{R}^2)).
\end{equation*}
\end{theorem}

\begin{theorem}[Maximal $L^2$-Regularity]\label{thm26}
Let $v$ be the solution of system \eqref{eq24} under the assumptions of Theorem \ref{thm25}. Let $C_\star = C_\star (\kappa ) >0$ be the constant appearing in \eqref{eq32}. Assume that $v_0 \in W^{1,2}(\mathbb{R}^2)$ and that 
\begin{equation*}
\Vert \nabla_X \varphi \Vert_{L^\infty (\mathbb{R}^2)} \leq \frac{1} {4 \sqrt{C_\star}}.
\end{equation*}
Then the four assertions hold:\\
\noindent $(\rm{i})$ There is $C_1 = C_1 (\kappa, \Vert \nabla_X^2 \varphi \Vert_{L^\infty (\mathbb{R}^2)} ) >0$ such that for each $T>0$
\begin{multline*}
\Vert G^{1/4} {dv}/{dt} \Vert_{L^2(0,T; L^2(\mathbb{R}^2))}+ \Vert G^{1/4} L v \Vert_{L^2(0,T; L^2(\mathbb{R}^2))}\\
 \leq C_1 \Vert v_0 \Vert_{\mathscr{W}^{1,2}(\Gamma)} + C_1 (1 + \sqrt{T}) \Vert G^{1/4} F \Vert_{L^2(0,T; L^2(\mathbb{R}^2))}.
\end{multline*}
\noindent $(\rm{ii})$ Assume in addition that there is $Q = { }^t (Q_1,Q_2,Q_3) \in L^2(0,\infty; W^{1,2}(\mathbb{R}^2))$ such that
\begin{equation*}
F = g^{\alpha \beta}g_\beta \cdot \frac{\partial Q}{\partial X_\alpha} \text{ on } (0,\infty),
\end{equation*}
and that
\begin{equation*}
Q \cdot g_1 \times g_2 =0 \text{ on } (0,\infty).
\end{equation*}
Then there is $C_2 = C_2 (\kappa , \Vert \nabla_X^2 \varphi \Vert_{L^\infty (\mathbb{R}^2)}  ) >0$ such that
\begin{multline}\label{eq27}
\Vert G^{1/4}{dv}/{dt} \Vert_{L^2(0,\infty; L^2(\mathbb{R}^2))}+ \Vert G^{1/4} L v \Vert_{L^2(0,\infty; L^2(\mathbb{R}^2))}\\
 \leq C_2 (\Vert v_0 \Vert_{\mathscr{W}^{1,2}(\Gamma)} + \Vert G^{1/4} F \Vert_{L^2(0,\infty; L^2(\mathbb{R}^2))} + \Vert G^{1/4} Q \Vert_{L^2(0,\infty; L^2(\mathbb{R}^2))}).
\end{multline}
\noindent $(\rm{iii})$ Assume in addition that $F \equiv 0$. Then 
\begin{equation}\label{eq28}
\Vert G^{1/4}{dv}/{dt} \Vert_{L^2(0,\infty; L^2(\mathbb{R}^2))}+ \Vert G^{1/4} L v \Vert_{L^2(0,\infty; L^2(\mathbb{R}^2))} \leq C_2 \Vert v_0 \Vert_{\mathscr{W}^{1,2}(\Gamma)}.
\end{equation}
Here $C_2$ is the positive constant appearing in \eqref{eq27}.\\
\noindent $(\rm{iv})$ Assume in addition that $F \in L^2 ( \mathbb{R}^2 \times (0,\infty) ) $. Then for each fixed $T \in (0,\infty)$
\begin{equation*}
v \in L^2 ( \mathbb{R}^2 \times (0,T) ).
\end{equation*}
\end{theorem}

\begin{theorem}[Stability]\label{thm27}
Let $v$ be the solution of system \eqref{eq24} under the assumptions of Theorem \ref{thm25}. Let $C_\star = C_\star (\kappa ) >0$ be the constant appearing in \eqref{eq32}. Assume that $F \equiv 0$ and that
\begin{equation*}
\Vert \nabla_X \varphi \Vert_{L^\infty (\mathbb{R}^2)} \leq \frac{1} {4 \sqrt{C_\star}}.
\end{equation*}
Then
\begin{equation}\label{eq29}
\lim_{ t \to \infty } \Vert G^{1/4} v (t) \Vert_{L^2(\mathbb{R}^2)} = 0 .
\end{equation}
\end{theorem}
\begin{remark}\label{rem28}
Applying \eqref{eq21}-\eqref{eq23}, we can write \eqref{eq26}, \eqref{eq27}, and \eqref{eq29} as follows:
\begin{equation*}
\Vert v(t) \Vert_{\mathscr{L}^2(\Gamma )}^2 + 2 \kappa \int_s^t \Vert v(\tau) \Vert_{\dot{\mathscr{W} }^{1,2}}^2{ \ }d \tau = \Vert  v(s ) \Vert_{\mathscr{L}^2(\Gamma )}^2 + \int_s^t \int_{\mathbb{R}^2} F v \sqrt{G} { \ }d X d \tau,
\end{equation*}
\begin{multline*}
\Vert {dv}/{dt} \Vert_{L^2(0,\infty; \mathscr{L}^2(\Gamma ))}+ \Vert L v \Vert_{L^2(0,\infty; \mathscr{L}^2(\Gamma))}\\
 \leq C_2 (\Vert v_0 \Vert_{\mathscr{W}^{1,2}(\Gamma)} + \Vert Q \Vert_{L^2(0,\infty; \mathscr{W}^{1,2}(\Gamma )}),
\end{multline*}
and
\begin{equation*}
\lim_{ t \to \infty } \Vert v (t) \Vert_{\mathscr{L}^2(\Gamma)} = 0 .
\end{equation*}
\end{remark}
Combining Lemma \ref{lem23}, \eqref{eq21}-\eqref{eq23}, and Theorems \ref{thm25}-\ref{thm27}\\ with $\delta_* = \min \{ 1/4, 1/(4 \sqrt{C_\star}) \}$, we have Theorem \ref{thm11}. We prove Theorem \ref{thm25} in Section \ref{sect4}, Theorem \ref{thm26} in Section \ref{sect5}, and Theorem \ref{thm27} in Section \ref{sect7}.

%\newpage

\section{Tools}\label{sect3}

In this section, we prepare some tools to derive our results. Let us now recall basic properties of the Laplace operator and the surface divergence theorem. Fix $\kappa >0$ and $\varphi \in BC^2(\mathbb{R}^2)$.

Let us first define the \emph{Laplace operator} $A$ on $L^2 (\mathbb{R}^2)$ by
\begin{equation}\label{eq31}
\begin{cases}
A f = - \kappa \Delta_X f,\\
D (A) = W^{2,2}(\mathbb{R}^2).
\end{cases}
\end{equation}
It is well-known that $A$ has the following nice properties.
\begin{lemma}[Properties of the Laplace operator]\label{lem31}{ \ }\\
$(\rm{i})$ The operator $-A$ generates a contraction $C_0$-semigroup on $L^2 (\mathbb{R}^2)$.\\
$(\rm{ii})$ The operator $-A$ generates a bounded analytic semigroup on $L^2 (\mathbb{R}^2)$.\\
$(\rm{iii})$ $D (A^{1/2}) = W^{1,2}(\mathbb{R}^2)$, and $\Vert A^{1/2} f \Vert_{L^2(\mathbb{R}^2)}^2 = \sqrt{\kappa} \Vert \nabla_X f \Vert_{L^2(\mathbb{R}^2)}^2$ for all $f \in D (A^{1/2})$.\\
$(\rm{iv})$ The operator $A$ has maximal $L^2$-regularity, that is, for each $V_0 \in D (A^{1/2})$ and $F \in L^2(0,\infty; L^2(\mathbb{R}^2))$, there exists a unique function $V$ satisfying system
\begin{equation*}
\begin{cases}
d V/{dt} + A V = F \text{ on }(0,\infty),\\
V \vert_{t=0} = V_0,
\end{cases}
\end{equation*}
and the estimate
\begin{multline}\label{eq32}
\Vert dV/{dt} \Vert_{L^2(0,\infty;L^2(\mathbb{R}^2)) } + \Vert A V \Vert_{L^2(0,\infty;L^2(\mathbb{R}^2)) }\\
 \leq C_\star (\Vert V_0 \Vert_{W^{1,2} (\mathbb{R}^2) } + \Vert F \Vert_{L^2(0,\infty;L^2(\mathbb{R}^2)) }),
\end{multline}
where $C_\star = C_\star (\kappa) >0$ independent of $(V_0,F)$. Moreover, if\\ $F \in C_{loc}^\eta ((0, \infty); L^2(\mathbb{R}^2) )$ for some $0 < \eta <1$, then
\begin{equation}\label{eq33}
V, dV/{dt}, A V \in C_{loc}^{ \eta }((0,\infty) ; L^2(\mathbb{R}^2)).
\end{equation}
$(\rm{v})$ Set $\mathcal{E}[t] = \mathcal{E} (X ,t) = {\rm{exp}}(- \vert X \vert^2/{4\kappa t})/{ 4 \pi \kappa t}$. Then ${\rm{e}}^{- tA} f = \mathcal{E}[t]*f$ for each $f \in L^2(\mathbb{R}^2)$.
\end{lemma}
\noindent See \cite[Section 5 in Chapter VI]{EN00} and \cite[Chapter 7]{Paz83} for the assertions $(\rm{i})$, $(\rm{ii})$, \cite[Chapter II.3]{Soh01} for $({\rm{iii}})$, and \cite{Des64}, \cite{DHP03}, \cite{KW04} for \eqref{eq32}. See Lemma \ref{lem82} in Appendix and \cite[Chapter 4]{Paz83} for \eqref{eq33}.

Next, we study the surface divergence theorem. From \cite{Sim83} and \cite{Kob20}, we have the lemma.
\begin{lemma}[Surface divergence theorem]\label{lem32}{ \ }\\
For every $\Phi = { }^t (\Phi_1,\Phi_2,\Phi_3) \in [W^{1,1} (\Gamma)]^3$,
\begin{equation}\label{eq34}
\int_\Gamma \nabla_\Gamma  \cdot \Phi { \ }d \mathcal{H}^2_x = - \int_{\Gamma} H_\Gamma ( n \cdot \Phi)  { \ }d \mathcal{H}_x^2.
\end{equation}
Here $H_\Gamma = -{\rm{div}}_\Gamma n$ is the mean curvature in the direction $n$, and
\begin{equation*}
W^{1,1}(\Gamma) = \{ \psi : \Gamma \to \mathbb{R}; { \ }\psi(\widehat{x}(X)) \in W^{1,1}(\mathbb{R}^2) \}.
\end{equation*}
\end{lemma}

Applying Lemma \ref{lem32}, we derive the following formulas.
\begin{lemma}\label{lem33}
For every $f \in W^{2,2} (\mathbb{R}^2)$ and $\phi, Q_1,Q_2,Q_3 \in W^{1,2} (\mathbb{R}^2)$,
\begin{equation}\label{eq35}
\int_{\mathbb{R}^2} \bigg\{ \frac{\partial}{\partial X_\alpha} \bigg( \sqrt{G} g^{\alpha \beta} \frac{\partial f }{\partial X_\beta} \bigg) \bigg\} \phi { \ }dX = - \int_{\mathbb{R}^2} g^{\alpha \beta} \frac{\partial f }{\partial X_\alpha} \frac{ \partial \phi }{\partial X_\beta} \sqrt{G} { \ }dX,
\end{equation}
and
\begin{multline}\label{eq36}
\int_{\mathbb{R}^2} \bigg( g^{\alpha \beta}g_\beta \cdot \frac{\partial Q }{\partial X_\alpha} \bigg) \phi \sqrt{G} { \ }dX\\ =- \int_{\mathbb{R}^2} Q \cdot g^{\alpha \beta} g_\alpha \frac{ \partial \phi }{\partial X_\beta} \sqrt{G} { \ }dX + \int_{\mathbb{R}^2} \bigg( g^{\alpha \beta}g_\beta \cdot \frac{\partial \widehat{n}}{\partial X_\alpha} \bigg) ( \widehat{n} \cdot Q ) \phi \sqrt{G} { \ }dX.
\end{multline}
Here $Q = Q(X) = { }^t (Q_1,Q_2,Q_3)$ and $\widehat{n} = \widehat{n}(X) = (g_1 \times g_2)/{\vert g_1 \times g_2 \vert}$.
\end{lemma}

\begin{proof}[Proof of Lemma \ref{lem33}]
Let $\psi, \psi_\sharp, \psi_\flat, \Psi_1, \Psi_2, \Psi_3 \in C^\infty ( \mathbb{R}^3)$. Write $\widehat{\psi} = \widehat{\psi}(X) =\psi (\widehat{x} (X))$ and $\Psi = { }^t (\Psi_1 , \Psi_2, \Psi_3)$. Assume that $\widehat{\psi}_\sharp \in W^{2,2} (\mathbb{R}^2)$ and $\widehat{\psi}, \widehat{\psi}_\flat, \widehat{\Psi} \in W^{1,2} (\mathbb{R}^2)$.

We first show \eqref{eq35}. Substituting $\Phi = \psi_\flat \nabla_\Gamma \psi_\sharp$ into \eqref{eq34}, we obtain
\begin{equation*}
\int_\Gamma (\Delta_\Gamma \psi_\sharp) \psi_\flat { \ }d \mathcal{H}^2_x = - \int_{\Gamma} \nabla_\Gamma \psi_\sharp \cdot \nabla_\Gamma \psi_\flat  { \ }d \mathcal{H}_x^2.
\end{equation*}
Note that $n \cdot \nabla_\Gamma \psi_\flat =0$. From Lemma \ref{lem22}, we have
\begin{align*}
\int_{\mathbb{R}^2} \bigg\{ \frac{\partial}{\partial X_\alpha} \bigg( \sqrt{G} g^{\alpha \beta} \frac{\partial \widehat{\psi}_\sharp }{\partial X_\beta} \bigg) \bigg\} \widehat{\psi}_\flat { \ }dX = - \int_{\mathbb{R}^2} g^{\alpha \beta} \frac{\partial \widehat{\psi}_\sharp }{\partial X_\alpha} \frac{ \partial \widehat{\psi}_\flat }{\partial X_\beta} \sqrt{G} { \ }dX.
\end{align*}
Write $f = \widehat{\psi}_\sharp$, $\phi = \widehat{\psi}_\flat$, then we see \eqref{eq35}.

Next, we derive \eqref{eq36}. Substituting $\Phi = \Psi \psi$ into \eqref{eq34}, we obtain
\begin{equation*}
\int_\Gamma ( \nabla_\Gamma \cdot \Psi) \psi { \ }d \mathcal{H}^2_x = -\int_{\Gamma} \Psi \cdot \nabla_\Gamma \psi { \ }d \mathcal{H}_x^2 - \int_{\Gamma} H_\Gamma ( n \cdot \Psi ) \psi  { \ }d \mathcal{H}_x^2.
\end{equation*}
Since $H_\Gamma = - {\rm{div}}_\Gamma n$, we use Lemma \ref{lem22} to find that
\begin{align*}
\int_{\mathbb{R}^2} \bigg( g^{\alpha \beta}g_\beta \cdot \frac{\partial \widehat{\Psi} }{\partial X_\alpha} \bigg) \widehat{\psi} \sqrt{G} { \ }dX = - \int_{\mathbb{R}^2} \widehat{\Psi} \cdot g^{\alpha \beta} g_\alpha \frac{ \partial \widehat{\psi} }{\partial X_\beta} \sqrt{G} { \ }dX\\
+ \int_{\mathbb{R}^2} \bigg\{ g^{\alpha \beta}g_\beta \cdot \frac{\partial }{\partial X_\alpha} \frac{g_1 \times g_2}{\vert g_1 \times g_2 \vert} \bigg\} \bigg( \frac{g_1 \times g_2}{\vert g_1 \times g_2 \vert} \cdot \widehat{\Psi} \bigg) \widehat{\psi} \sqrt{G} { \ }dX.
\end{align*}
Write $Q = \widehat{\Psi}$ and $\phi = \widehat{\psi}$. Then we have \eqref{eq36}. Therefore, the lemma follows.
  \end{proof}
\noindent Indeed, we can derive \eqref{eq35} and \eqref{eq36} by using the usual integration by parts formula.

Finally, we prepare one useful lemma for the reader.
\begin{lemma}\label{lem34}
For all $f \in W^{2,2} (\mathbb{R}^2),$
\begin{equation}\label{eq37}
\bigg\Vert \frac{\partial^2 f}{\partial X_1^2} \bigg\Vert_{L^2(\mathbb{R}^2)}^2 + \bigg\Vert \frac{\partial^2 f}{\partial X_2^2} \bigg\Vert_{L^2(\mathbb{R}^2)}^2 + 2 \bigg\Vert \frac{\partial^2 f}{\partial X_1 \partial X_2} \bigg\Vert_{L^2(\mathbb{R}^2)}^2 = \Vert \Delta_X f \Vert_{L^2(\mathbb{R}^2)}^2.
\end{equation}
\end{lemma}

\begin{proof}[Proof of Lemma \ref{lem34}]
Let $f \in W^{2,2} (\mathbb{R}^2)$. Since $C_0^\infty ( \mathbb{R}^2)$ is dense in $W^{2,2} (\mathbb{R}^2)$, we take $\{ f_m \}_{m \in \mathbb{N}} \subset C_0^\infty (\mathbb{R}^2)$ such that
\begin{equation*}
\lim_{m \to \infty } \Vert f_m - f \Vert_{W^{2,2} (\mathbb{R}^2)} = 0 .
\end{equation*}
By integration by parts, we check that
\begin{multline*}
\Vert \Delta_X f_m \Vert_{L^2(\mathbb{R}^2)}^2  = \left<\frac{\partial^2 f_m}{\partial X_1^2} + \frac{\partial^2 f_m}{\partial X_2^2} , \frac{\partial^2 f_m}{\partial X_1^2} + \frac{\partial^2 f_m}{\partial X_2^2}  \right>_{L^2(\mathbb{R}^2)}\\
 = \bigg\Vert \frac{\partial^2 f_m}{\partial X_1^2} \bigg\Vert_{L^2(\mathbb{R}^2)}^2 + \bigg\Vert \frac{\partial^2 f_m}{\partial X_2^2} \bigg\Vert_{L^2(\mathbb{R}^2)}^2 + 2 \left<\frac{\partial^2 f_m}{\partial X_1^2}  , \frac{\partial^2 f_m}{\partial X_2^2}  \right>_{L^2(\mathbb{R}^2)}\\
 = \bigg\Vert \frac{\partial^2 f_m}{\partial X_1^2} \bigg\Vert_{L^2(\mathbb{R}^2)}^2 + \bigg\Vert \frac{\partial^2 f_m}{\partial X_2^2} \bigg\Vert_{L^2(\mathbb{R}^2)}^2 + 2 \bigg\Vert \frac{\partial^2 f_m}{\partial X_1 \partial X_2} \bigg\Vert_{L^2(\mathbb{R}^2)}^2 .
\end{multline*}
Letting $m \to \infty$, we have \eqref{eq37}. Therefore, the lemma follows.
  \end{proof}

%\newpage

\section{Laplace-Beltrami Operator on Unbounded Surfaces}\label{sect4}

In this section, we study the Laplace-Beltrami operator on our unbounded surface to prove Theorem \ref{thm25}. Fix $\kappa >0$ and $\varphi \in BC^2(\mathbb{R}^2)$. Let $L$ and $A$ be the two operators defined by \eqref{eq25} and \eqref{eq31}, respectively.

Let us now derive the properties of the operator $L$.
\begin{proposition}\label{prop41}Assume that $\Vert \nabla_X \varphi \Vert_{L^\infty (\mathbb{R}^2)} \leq 1/4$. Then\\
$(\rm{i})$ The operator $-L$ generates a $C_0$-semigroup on $L^2 (\mathbb{R}^2)$.\\
$(\rm{ii})$ The operator $-L$ generates an analytic semigroup on $L^2 (\mathbb{R}^2)$.\\
$(\rm{iii})$ For all $f \in L^2(\mathbb{R}^2)$ and $t > s \geq 0$
\begin{equation}\label{eq41}
\Vert {\rm{e}}^{- t L} f \Vert_{L^2(\mathbb{R}^2) } \leq 3^{1/4} \Vert {\rm{e}}^{- s L} f \Vert_{L^2(\mathbb{R}^2) } .
\end{equation}
$(\rm{iv})$ For all $f \in L^2(\mathbb{R}^2)$, 
\begin{equation}\label{eq42}
\int_0^\infty \Vert \nabla_X {\rm{e}}^{- \tau L} f \Vert_{L^2(\mathbb{R}^2) }^2 { \ }d \tau \leq \frac{3}{2 \kappa} \Vert f \Vert_{L^2(\mathbb{R}^2) }^2.
\end{equation}

\end{proposition}

To prove Proposition \ref{prop41}, we prepare two lemmas.
\begin{lemma}\label{lem42}
Assume that $\Vert \nabla_X \varphi \Vert_{L^\infty (\mathbb{R}^2)} \leq 1$. Then\\
\noindent $(\rm{i})$ For all $X \in \mathbb{R}^2$,
\begin{equation}\label{eq43}
1 \leq G(X) \leq 3.
\end{equation}
\noindent $(\rm{ii})$ For all $f \in W^{1,2}(\mathbb{R}^2)$,
\begin{equation}\label{eq44}
\int_{\mathbb{R}^2} g^{\alpha \beta} \frac{\partial f }{\partial X_\alpha} \frac{ \partial f }{\partial X_\beta} \sqrt{G} { \ }dX \geq \frac{1}{\sqrt{3}} \Vert \nabla_X f \Vert_{L^2(\mathbb{R}^2)}^2,
\end{equation}
\begin{equation}\label{eq45}
\int_{\mathbb{R}^2} g^{\alpha \beta} \frac{\partial f }{\partial X_\alpha} \frac{ \partial f }{\partial X_\beta} \sqrt{G} { \ }dX \leq 3 \Vert \nabla_X f \Vert_{L^2(\mathbb{R}^2)}^2.
\end{equation}
\end{lemma}

\begin{proof}[Proof of Lemma \ref{lem42}]
Assume that $\Vert \nabla_X \varphi \Vert_{L^\infty (\mathbb{R}^2)} \leq 1$. From Lemma \ref{lem21}, we see $(\rm{i})$. Let $f \in W^{1,2} (\mathbb{R}^2)$. By the definitions of $g^{\alpha \beta}$, $g_{\alpha \beta}$, and \eqref{eq43}, we check that
\begin{multline*}
\text{(L.H.S.) of }\eqref{eq44}\\
 = \int_{\mathbb{R}^2}\bigg( g^{11}\frac{\partial f}{\partial X_1} \frac{\partial f}{\partial X_1} + g^{22}\frac{\partial f}{\partial X_2} \frac{\partial f}{\partial X_2} + 2 g^{12}\frac{\partial f}{\partial X_1} \frac{\partial f}{\partial X_2} \bigg) \sqrt{G} { \ }dX\\
= \int_{\mathbb{R}^2} \frac{1}{\sqrt{G}} \bigg( g_{22} \frac{\partial f}{\partial X_1} \frac{\partial f}{\partial X_1} +g_{11} \frac{\partial f}{\partial X_2} \frac{\partial f}{\partial X_2} -2 g_{12} \frac{\partial f}{\partial X_1} \frac{\partial f}{\partial X_2} \bigg) { \ }d X\\ 
= \int_{\mathbb{R}^2} \frac{1}{\sqrt{G}} \vert \nabla_X f \vert^2 { \ }d X
+ \int_{\mathbb{R}^2} \frac{1}{\sqrt{G}} \bigg( \frac{\partial \varphi}{\partial X_2} \frac{\partial f}{\partial X_1} - \frac{\partial \varphi}{\partial X_1} \frac{\partial f}{\partial X_2} \bigg)^2 { \ }d X\\ 
 \geq \frac{1}{\sqrt{3}} \Vert \nabla_X f \Vert_{L^2(\mathbb{R}^2)}^2,
\end{multline*}
and that
\begin{multline*}
\text{(L.H.S.) of }\eqref{eq45}\\
\leq \int_{\mathbb{R}^2} \frac{1}{\sqrt{G}} \bigg( \{ 1+2 (\partial_{X_2} \varphi )^2 \} \frac{\partial f}{\partial X_1} \frac{\partial f}{\partial X_1} +  \{1+2 (\partial_{X_1} \varphi )^2 \} \frac{\partial f}{\partial X_2} \frac{\partial f}{\partial X_2} \bigg) { \ }d X\\ 
 \leq 3 \Vert \nabla_X f \Vert_{L^2(\mathbb{R}^2)}^2.
\end{multline*}
Here we used the facts that $(a-b)^2 \leq 2 a^2 + 2 b^2$ and $\Vert \nabla_X \varphi \Vert_{L^\infty (\mathbb{R}^2)} \leq 1$. Therefore, we see $(\rm{ii})$.
  \end{proof}

\begin{lemma}\label{lem43}
Set
\begin{equation}\label{eq46}
\begin{cases}
B f = (L- A)f,\\
D (B) = D(A) (= W^{2,2}(\mathbb{R}^2)).
\end{cases}
\end{equation}
Then for all $f \in D (A)$
\begin{multline}\label{eq47}
\Vert Bf \Vert_{L^2 ( \mathbb{R}^2)}\\
 \leq \Vert \nabla_X \varphi \Vert_{L^\infty ( \mathbb{R}^2)}^2 \Vert A f \Vert_{L^2 ( \mathbb{R}^2)} + 40\kappa \Vert \nabla_X^2 \varphi \Vert_{L^{\infty}(\mathbb{R}^2)} \Vert \nabla_X f \Vert_{L^2 (\mathbb{R}^2)}.
\end{multline}
\end{lemma}

\begin{proof}[Proof of Lemma \ref{lem43}]
Fix $f \in D (A)$. From \eqref{eq25}, we find that
\begin{equation*}
L f = \kappa L_1 f + \kappa L_2 f + \kappa L_3 f,
\end{equation*}
where
\begin{equation*}
L_1 f  = - g^{\alpha \beta} \frac{\partial^2 f }{\partial X_\alpha \partial X_\beta},{ \ }L_2 f = - \frac{\partial g^{\alpha \beta} }{\partial X_\alpha} \frac{\partial f}{\partial X_\beta},{ \ }
L_3 f = - \frac{g^{\alpha \beta}}{ 2  G}\frac{\partial G }{\partial X_\alpha} \frac{\partial f}{\partial X_\beta} .
\end{equation*}
Using $1/G \leq 1$, the H\"{o}lder inequality, and \eqref{eq37}, we find that
\begin{multline}
\Vert - L_1 f - \Delta_X f  \Vert_{L^2(\mathbb{R}^2)}\\
  \leq \bigg\Vert \frac{(\partial_{ X_1 } \varphi)^2}{G} \frac{\partial^2 f }{\partial X_1 \partial X_1}+\frac{(\partial_{ X_2 } \varphi)^2}{G} \frac{\partial^2 f }{\partial X_2 \partial X_2} - \frac{2(\partial_{X_1} \varphi)(\partial_{X_2} \varphi)}{G} \frac{\partial^2 f }{\partial X_1 \partial X_2} \bigg\Vert_{L^2(\mathbb{R}^2)}\\
 \leq \Vert \nabla_X \varphi \Vert_{L^\infty(\mathbb{R}^2)}^2 \Vert \Delta_X f \Vert_{L^2(\mathbb{R}^2)}.\label{eq48}
\end{multline}
Since
\begin{equation*}
G = G (X) = 1 + (\partial_{X_1} \varphi )^2 + (\partial_{X_2} \varphi)^2,
\end{equation*}
we easily check that for all $X \in \mathbb{R}^2$,
\begin{align}
\label{eq49} \frac{1}{ G } \leq 1,{ \ }\frac{ \vert \partial_{X_1} \varphi \vert}{ G} \leq 1,{ \ }\frac{\vert \partial_{X_2} \varphi \vert }{G} \leq 1,{ \ }\frac{ \vert \partial_{X_1} \varphi \vert^2}{ G} \leq 1,{ \ }\frac{\vert \partial_{X_2} \varphi \vert^2 }{G} \leq 1,\\
\label{EQ4010} \frac{ \vert \partial_{X_1} \varphi \vert \vert \partial_{X_2} \varphi \vert }{ G} \leq 1,{ \ }\frac{ 1 + \vert \partial_{X_1} \varphi \vert^2 }{ G} \leq 1,{ \ }\frac{ 1+ \vert \partial_{X_2} \varphi \vert^2 }{ G} \leq 1.
\end{align}
From the definitions of $g^{\alpha \beta}$, $g_{\alpha \beta}$, \eqref{eq49}, and \eqref{EQ4010}, we see that for each $\alpha , \beta \in \{ 1,2 \}$,
\begin{align}
\label{EQ4011} \bigg\Vert \frac{\partial g^{\alpha \beta}}{\partial X_\alpha} \bigg\Vert_{L^\infty (\mathbb{R}^2)} & \leq 6 \Vert \nabla_X^2 \varphi \Vert_{L^\infty (\mathbb{R}^2)},\\
\label{EQ4012} \bigg\Vert \frac{g^{\alpha \beta}}{G} \frac{\partial G}{\partial X_\alpha} \bigg\Vert_{L^\infty (\mathbb{R}^2)} & \leq 4 \Vert \nabla_X^2 \varphi \Vert_{L^\infty (\mathbb{R}^2)}.
\end{align}
Using \eqref{EQ4011} and \eqref{EQ4012}, we observe that
\begin{align}
\Vert L_2 f \Vert_{L^2(\mathbb{R}^2)} & = \bigg\Vert \frac{1}{2}  \sum_{\alpha,\beta=1}^2 \frac{\partial g^{\alpha \beta} }{\partial X_\alpha} \frac{\partial f}{\partial X_\beta} \bigg\Vert_{L^2(\mathbb{R}^2)}\notag\\
& \leq 24 \Vert \nabla_X^2 \varphi \Vert_{L^{\infty}(\mathbb{R}^2)}  \Vert \nabla_X f \Vert_{L^2(\mathbb{R}^2)},\label{EQ4013}
\end{align}
and that
\begin{align}
\Vert L_3 f \Vert_{L^2 ( \mathbb{R}^2 )} &= \bigg\Vert \sum_{\alpha,\beta=1}^2 \frac{g^{\alpha \beta}}{G}\frac{\partial G }{\partial X_\alpha} \frac{\partial f}{\partial X_\beta} \bigg\Vert_{L^2(\mathbb{R}^2)}\notag\\
& \leq 16 \Vert \nabla_X^2 \varphi \Vert_{L^{\infty}(\mathbb{R}^2)} \Vert \nabla_X f \Vert_{L^2(\mathbb{R}^2)}. \label{EQ4014}
\end{align}
By \eqref{eq48}, \eqref{EQ4013}, and \eqref{EQ4014}, we check that
\begin{align*}
\Vert Bf \Vert_{L^2 ( \mathbb{R}^2)} & = \Vert - (\kappa L_1 + A )f - \kappa L_2 f - \kappa L_3 f \Vert_{L^2 ( \mathbb{R}^2)}\\
& \leq \Vert \nabla_X \varphi \Vert_{L^\infty ( \mathbb{R}^2)}^2 \Vert A f \Vert_{L^2 ( \mathbb{R}^2)} + 40 \kappa \Vert \nabla_X^2 \varphi \Vert_{L^{\infty}(\mathbb{R}^2)}  \Vert\nabla_X f \Vert_{L^2 (\mathbb{R}^2)}.
\end{align*}
Thus, we have \eqref{eq47}. Therefore, the lemma follows.
  \end{proof}

Let us prove Proposition \ref{prop41}.
\begin{proof}[Proof of Proposition \ref{prop41}]

Assume that $\Vert \nabla_X \varphi \Vert_{L^\infty(\mathbb{R}^2)} \leq 1/4$. Using \eqref{eq47}, the interpolation theory, and \eqref{eq37}, we observe that for all $f \in D(A)$
\begin{align}
\Vert Bf \Vert_{L^2 ( \mathbb{R}^2)} & \leq \frac{1}{16} \Vert A f \Vert_{L^2 ( \mathbb{R}^2)} + 40 \kappa \Vert \nabla_X^2 \varphi \Vert_{L^{\infty}(\mathbb{R}^2)} \Vert \nabla_X f \Vert_{L^2 (\mathbb{R}^2)}\notag\\
& \leq \frac{1}{4} \Vert A f \Vert_{L^2 ( \mathbb{R}^2)} + C ( \kappa , \Vert \nabla_X^2 \varphi \Vert_{L^{\infty}(\mathbb{R}^2)} ) \Vert  f \Vert_{L^2 (\mathbb{R}^2)}.\label{EQ4015}
\end{align}
Since $-A$ generates an analytic contraction semigroup on $L^2 ( \mathbb{R}^2)$, it follows from the perturbation theory on the semigroup theory(\cite[Chapter 3]{Paz83}, \cite[Chapter III]{EN00}) to find $(\rm{i})$ and $(\rm{ii})$.

Now we show $(\rm{iii})$ and $(\rm{iv})$. Let $f \in L^2 ( \mathbb{R}^2 )$. Set
\begin{equation*}
w(t) = {\rm{e}}^{- t L} f{ \ }(t>0).
\end{equation*}
Since $-L$ generates an analytic semigroup on $L^2(\mathbb{R}^2)$, we see that
\begin{equation*}
w \in C ([0,\infty);L^2(\mathbb{R}^2)) \cap C ((0,\infty); W^{2,2} (\mathbb{R}^2)) \cap C^1 ((0,\infty); L^2(\mathbb{R}^2)),
\end{equation*}
and that $w$ satisfies
\begin{equation}\label{EQ4016}
\begin{cases}
dw/{dt} + L w = 0 \text{ on } (0,\infty),\\
w\vert_{t=0} = f.
\end{cases}
\end{equation}
Using \eqref{EQ4016} and \eqref{eq35}, we check that for $\tau >0$
\begin{align*}
\frac{1}{2} \frac{d}{d\tau} \Vert G^{1/4} w(\tau) \Vert_{L^2(\mathbb{R}^2)}^2 & = \int_{\mathbb{R}^2} \{ L w(\tau) \} w(\tau) \sqrt{G} { \ }d X\\
& = - \kappa \int_{\mathbb{R}^2} g^{\alpha \beta} \frac{\partial w }{\partial X_\alpha} \frac{ \partial w }{\partial X_\beta} \sqrt{G}{ \ }dX.
\end{align*}
Integrating with respect to $\tau$, we see that for $0 \leq s <t$
\begin{equation*}
\Vert G^{1/4} w(t) \Vert_{L^2(\mathbb{R}^2)}^2 + 2 \kappa \int_s^t \int_{\mathbb{R}^2} g^{\alpha \beta} \frac{\partial w }{\partial X_\alpha} \frac{ \partial w }{\partial X_\beta} \sqrt{G} { \ }dX d\tau = \Vert G^{1/4} w(s) \Vert_{L^2(\mathbb{R}^2)}^2. 
\end{equation*}
Applying Lemma \ref{lem42}, we check that for all $t > s \geq 0$,
\begin{equation*}
\Vert w(t) \Vert_{L^2(\mathbb{R}^2)}^2 + \frac{2}{ \sqrt{3}} \kappa \int_s^t \Vert \nabla_X w ( \tau ) \Vert_{L^2(\mathbb{R}^2)}^2 { \ } d \tau \leq \sqrt{3} \Vert w (s) \Vert_{L^2(\mathbb{R}^2)}^2. 
\end{equation*}
Since $w (t) = {\rm{e}}^{- t L} f$, we see $(\rm{iii})$ and that for $t >0$
\begin{equation*}
\int_0^t \Vert \nabla_X w ( \tau ) \Vert_{L^2(\mathbb{R}^2)}^2 { \ } d \tau \leq \frac{3}{2 \kappa} \Vert f \Vert_{L^2(\mathbb{R}^2)}^2. 
\end{equation*}
This implies $(\rm{iv})$. Therefore, Proposition \ref{prop41} is proved.
  \end{proof}

Let us prove Theorem \ref{thm25}.
\begin{proof}[Proof of Theorem \ref{thm25}]
Let $v_0 \in L^2( \mathbb{R}^2 )$ and\\
 $F \in L^2(0,\infty ;L^2(\mathbb{R}^2)) \cap C_{loc}^\eta ((0, \infty );L^2(\mathbb{R}^2))$ for some $0 < \eta < 1$. Assume that $\Vert \nabla_X \varphi \Vert_{L^\infty (\mathbb{R}^2)} < 1/4$. Since $- L$ generates an analytic semigroup on $L^2 ( \mathbb{R}^2)$ form Proposition \ref{prop41}, we see that system
\begin{equation}\label{EQ4017}
\begin{cases}
dv/{dt} + L v = F \text{ on } (0, \infty),\\
v \vert_{t = 0} = v_0,
\end{cases}
\end{equation}
admits a unique global-in-time strong solution $v$ in
\begin{equation}\label{EQ4018}
C( [0,\infty) ; L^2(\mathbb{R}^2)) \cap C( (0,\infty) ; W^{2,2}(\mathbb{R}^2)) \cap C^1( (0,\infty) ; L^2(\mathbb{R}^2)),
\end{equation}
satisfying that $v$ is represented by
\begin{equation}\label{EQ4019}
v (t) = {\rm{e}}^{ - t L} v_0 + \int_0^t {\rm{e}}^{- ( t - \tau ) L} F (\tau ) { \ }d \tau { \ }(t > 0).
\end{equation}
By the same argument to derive \eqref{eq41} and \eqref{eq42} in the proof of Proposition \ref{prop41}, we see that for all $0 \leq s < t <+ \infty$
\begin{multline*}
\Vert G^{1/4} v(t) \Vert_{L^2(\mathbb{R}^2 )}^2 + 2 \kappa \int_s^t \int_{\mathbb{R}^2} g^{\alpha \beta} \frac{\partial v(\tau) }{\partial X_\alpha} \frac{ \partial v(\tau) }{\partial X_\beta} \sqrt{G} { \ }dX d \tau\\
 = \Vert G^{1/4} v(s ) \Vert_{L^2(\mathbb{R}^2 )}^2 + \int_s^t \int_{\mathbb{R}^2} F(\tau) v(\tau) G^{1/2} { \ }d X d \tau,
\end{multline*}
which is \eqref{eq26}.

Finally, we state the H\"{o}lder continuity of the solution $v$. Since\\ $F \in L^2(0,\infty ;L^2(\mathbb{R}^2)) \cap C_{loc}^\eta ((0, \infty );L^2(\mathbb{R}^2))$, it follows from Lemma \ref{lem82} in Appendix to find that
\begin{equation*}
v, Lv, dv/{dt} \in C_{loc}^{ \eta } ((0, \infty);L^2 (\mathbb{R}^2) ).
\end{equation*}
Since
\begin{equation*}
\Vert G^{1/4} v (t) \Vert_{L^2(\mathbb{R}^2)} \leq \sqrt{3} \Vert v (t) \Vert_{L^2(\mathbb{R}^2)},
\end{equation*}
we conclude that
\begin{equation*}
G^{1/4} v, G^{1/4} Lv, G^{1/4} dv/{dt} \in C_{loc}^{ \eta } ((0, \infty);L^2 (\mathbb{R}^2) ).
\end{equation*}
Therefore, Theorem \ref{thm25} is proved.
  \end{proof}

%\newpage

\section{Maximal $L^2$-Regularity}\label{sect5}

In this section, we study maximal regularity of solutions to our diffusion system. We derive maximal $L^2$-regularity of our Laplace-Beltrami operator by applying maximal $L^2$-regularity of the Laplace operator. Let $\kappa >0$, $\varphi \in BC^2(\mathbb{R}^2)$, and let $C_\star = C_\star (\kappa ) >0$ be the constant appearing in \eqref{eq32}. Let $L$, $A$, and $B$ be the three operators defined by \eqref{eq25}, \eqref{eq31}, and \eqref{eq46}, respectively.
\begin{proof}[Proof of Theorem \ref{thm26}]
Let $v_0 \in W^{1,2} (\mathbb{R}^2)$, and\\
 $F \in L^2(0,\infty; L^2(\mathbb{R}^2)) \cap C_{loc}^\eta ((0,\infty); L^2 (\mathbb{R}^2))$ for some $0< \eta <1$. Assume that 
\begin{equation}\label{eq51}
\Vert \nabla_X \varphi \Vert_{L^\infty (\mathbb{R}^2)} \leq \min \left\{ \frac{1}{4} , \frac{1} {4 \sqrt{C_\star}} \right\}.
\end{equation}
From Theorem \ref{thm25} we find that system \eqref{EQ4017} admits a unique global-in-time strong solution $v$ in \eqref{EQ4018}, satisfying \eqref{EQ4019}. Write
\begin{equation*}
\mathcal{M} = \Vert \nabla_X^2 \varphi \Vert_{L^\infty (\mathbb{R}^2)}.
\end{equation*}
Now we admit that for each fixed $T>0$
\begin{equation}\label{eq52}
B v \in L^2(0,T;L^2(\mathbb{R}^2)).
\end{equation}
See the latter part for the proof of \eqref{eq52}.

We first show assertion $(\rm{i})$.  Since $v$ satisfies
\begin{equation*}
\begin{cases}
dv/{dt}  + Av = F - B v  \text{ on }(0,T),\\
v\vert_{t=0} = v_0,
\end{cases}
\end{equation*}
we apply \eqref{eq32} to have
\begin{multline}\label{eq53}
\Vert dv/{dt} \Vert_{L^2(0,T; L^2(\mathbb{R}^2))}+ \Vert A v \Vert_{L^2(0,T; L^2(\mathbb{R}^2))}\\
 \leq C_\star (\Vert v_0 \Vert_{W^{1,2}(\mathbb{R}^2)} + \Vert F \Vert_{L^2(0,T; L^2(\mathbb{R}^2))}+ \Vert Bv \Vert_{L^2(0,T; L^2(\mathbb{R}^2))}).
\end{multline}
Using \eqref{eq47} and \eqref{eq51}, we observe that
\begin{multline}\label{eq54}
C_\star \Vert B v \Vert_{L^2(0,T; L^2(\mathbb{R}^2))}\\
 \leq 2C_\star \Vert \nabla_X \varphi \Vert_{L^\infty (\mathbb{R}^2)}^2 \Vert A v \Vert_{L^2(0,T; L^2(\mathbb{R}^2))} + C(\kappa , \mathcal{M})  \Vert \nabla_X v \Vert_{L^2(0,T; L^2(\mathbb{R}^2))}\\
\leq (1/8) \Vert A v \Vert_{L^2(0,T; L^2(\mathbb{R}^2))} + C(\kappa , \mathcal{M}) \Vert \nabla_X v \Vert_{L^2(0,T; L^2(\mathbb{R}^2))}.
\end{multline}
From \eqref{eq53} and \eqref{eq54}, we have
\begin{multline}\label{eq55}
\Vert dv/{dt} \Vert_{L^2(0,T; L^2(\mathbb{R}^2))}+(7/8) \Vert A v \Vert_{L^2(0,T; L^2(\mathbb{R}^2))}\\
 \leq C_\star (\Vert v_0 \Vert_{W^{1,2}(\mathbb{R}^2)} + \Vert F \Vert_{L^2(0,T; L^2(\mathbb{R}^2))} ) \\
 + C (\kappa , \mathcal{M}) \Vert \nabla_X v \Vert_{L^2(0,T; L^2(\mathbb{R}^2))}.
\end{multline}
Since $L = A +B$, we use \eqref{eq54} and \eqref{eq55} to find that
\begin{multline}\label{eq56}
\Vert dv/{dt} \Vert_{L^2(0,T; L^2(\mathbb{R}^2))} + \Vert L v \Vert_{L^2(0,T; L^2(\mathbb{R}^2))}\\
 \leq C (\kappa , \mathcal{M}) (\Vert v_0 \Vert_{W^{1,2}(\mathbb{R}^2)} + \Vert F \Vert_{L^2(0,T; L^2(\mathbb{R}^2))}  + \Vert \nabla_X v \Vert_{L^2(0,T; L^2(\mathbb{R}^2))}).
\end{multline}
Now we consider $\Vert \nabla_X v \Vert_{L^2(0,T; L^2(\mathbb{R}^2))}$. By the same argument as in the proof of Proposition \ref{prop41}, we see that for all $t>0$
\begin{multline*}
\Vert G^{1/4} v(t)  \Vert_{L^2(\mathbb{R}^2)}^2 + 2 \kappa \int_0^t \int_{\mathbb{R}^2}g^{\alpha \beta} \frac{\partial v }{\partial X_\alpha} \frac{ \partial v }{\partial X_\beta} \sqrt{G} { \ }dX d\tau\\
 = \Vert G^{1/4} v_0 \Vert_{L^2(\mathbb{R}^2)}^2 + \int_0^t \left< F ,  v\sqrt{G} \right>_{L^2(\mathbb{R}^2)} { \ }d \tau. 
\end{multline*}
Using Lemma \ref{lem42}, we observe that
\begin{multline}\label{eq57}
\Vert v(t) \Vert_{L^2(\mathbb{R}^2)}^2 + \frac{2}{\sqrt{3}} \kappa \int_0^t \Vert \nabla_X v(\tau) \Vert_{L^2(\mathbb{R}^2)}^2 { \ } d\tau\\
 \leq \sqrt{3} \Vert v_0 \Vert_{L^2(\mathbb{R}^2)}^2 + \bigg\vert \int_0^t \left< F ,  v\sqrt{G}\right>_{L^2(\mathbb{R}^2)} { \ }d \tau \bigg\vert.
\end{multline}
This implies that
\begin{multline}\label{eq58}
\int_0^T \Vert \nabla_X v(\tau) \Vert_{L^2(\mathbb{R}^2)}^2 { \ } d\tau\\
 \leq \frac{3}{2 \kappa} \Vert v_0 \Vert_{L^2(\mathbb{R}^2)}^2 + \frac{\sqrt{3}}{2 \kappa} \bigg\vert \int_0^T \left< F ,  v\sqrt{G}\right>_{L^2(\mathbb{R}^2)} { \ }d \tau \bigg\vert.
\end{multline}
Using the Cauchy-Schwarz inequality with \eqref{eq41} and \eqref{EQ4019}, we see that for all $0 <t <T$,
\begin{align*}
\Vert v (t) \Vert_{L^2( \mathbb{R}^2)} & \leq \Vert {\rm{e}}^{-t L} v_0 \Vert_{L^2(\mathbb{R}^2)} + \int_0^t \Vert {\rm{e}}^{- (t- \tau ) L} F (\tau) \Vert_{L^2(\mathbb{R}^2)} { \ }d\tau\\
& \leq C \Vert v_0 \Vert_{L^2 (\mathbb{R}^2)} + C T^{1/2} \Vert F \Vert_{L^2(0,T;L^2(\mathbb{R}^2))}.
\end{align*}
This implies that
\begin{equation}\label{eq59}
\sup_{0 < t < T} \Vert v (t) \Vert_{L^2 (\mathbb{R}^2)} \leq C \Vert v_0 \Vert_{L^2 (\mathbb{R}^2)} + C T^{1/2} \Vert F \Vert_{L^2(0,T;L^2(\mathbb{R}^2))}.
\end{equation}
Applying the H\"{o}lder inequality, $ab \leq a^2/2 + b^2/2$, and \eqref{eq59}, we find that
\begin{multline}\label{EQ5010}
\bigg\vert \int_0^T \left< F , v \sqrt{G} \right>_{L^2(\mathbb{R}^2)} { \ }d\tau \bigg\vert \leq C \Vert v \Vert_{L^\infty(0,T;L^2( \mathbb{R}^2))} \Vert F \Vert_{L^2(0,T; L^2( \mathbb{R}^2) ) } T^{1/2} \\
 \leq C \Vert v_0 \Vert_{L^2 (\mathbb{R}^2)}^2 + C T \Vert F \Vert_{L^2(0,T;L^2(\mathbb{R}^2))}^2.
\end{multline}
From \eqref{eq58} and \eqref{EQ5010}, we observe that
\begin{equation}\label{EQ5011}
\Vert \nabla_X v \Vert_{L^2 (0,T ; L^2(\mathbb{R}^2))} \leq C \Vert v_0 \Vert_{L^2 (\mathbb{R}^2)} + C T^{1/2} \Vert F \Vert_{L^2(0,T;L^2(\mathbb{R}^2))}.
\end{equation}
By \eqref{eq56} and \eqref{EQ5011}, we have
\begin{multline*}
\Vert dv/{dt} \Vert_{L^2(0,T; L^2(\mathbb{R}^2))}+ \Vert L v \Vert_{L^2(0,T; L^2(\mathbb{R}^2))}\\
\leq C(\kappa , \mathcal{M})  (\Vert v_0 \Vert_{W^{1,2}(\mathbb{R}^2)} + (1 + \sqrt{T}) \Vert F \Vert_{L^2(0,T; L^2(\mathbb{R}^2))}).
\end{multline*}
From Lemma \ref{lem24}, we see $(\rm{i})$ if \eqref{eq52} holds.

Next, we show $(\rm{ii})$ and $(\rm{iii})$. In the proof of $(\rm{ii})$, we assume that there is\\ $Q = { }^t (Q_1,Q_2,Q_3) \in L^2(0,\infty; W^{1,2}(\mathbb{R}^2))$ such that $F = g^{\alpha \beta}g_\beta \cdot \frac{\partial Q}{\partial X_\alpha}$ and that $Q \cdot (g_1 \times g_2) =0$. By the same argument as in the proof of $(\rm{i})$, we see that for each $T>0$
\eqref{eq56} and \eqref{eq58} hold. Since $F = g^{\alpha \beta}g_\beta \cdot \frac{\partial Q}{\partial X_\alpha}$ and $Q \cdot (g_1 \times g_2) =0$, we use \eqref{eq36} to check that
\begin{align*}
\bigg\vert \int_0^T \left< F ,  v\sqrt{G}\right>_{L^2(\mathbb{R}^2)} { \ }d \tau \bigg\vert &= \bigg\vert \int_0^T \int_{\mathbb{R}^2} \bigg( g^{\alpha \beta}g_\beta \cdot \frac{\partial Q }{\partial X_\alpha} \bigg) v \sqrt{G} { \ }dX d \tau \bigg\vert\\
& = \bigg\vert - \int_0^T \int_{\mathbb{R}^2} Q \cdot g^{\alpha \beta} g_\alpha \frac{ \partial v }{\partial X_\beta} \sqrt{G} { \ }dX d\tau \bigg\vert.
\end{align*}
Using the Cauchy-Schwarz inequality and $2ab \leq a^2 +b^2$, we see that
\begin{align}
\bigg\vert \int_0^T \left< F ,  v\sqrt{G}\right>_{L^2(\mathbb{R}^2)} { \ }d \tau \bigg\vert \leq C \Vert Q \Vert_{L^2(0,T; L^2(\mathbb{R}^2))}\Vert \nabla_X v \Vert_{L^2(0,T; L^2(\mathbb{R}^2))}\notag\\
\leq \frac{1}{8} \Vert \nabla_X v \Vert_{L^2(0,T; L^2(\mathbb{R}^2))}^2 + C \Vert Q \Vert_{L^2(0,T; L^2(\mathbb{R}^2))}^2.\label{EQ5012}
\end{align}
Note that $\Vert g^{\alpha \beta} g_\beta \Vert_{L^\infty (\mathbb{R}^2)} \leq C$ since $\Vert \nabla_X \varphi \Vert_{L^\infty} \leq 1/4$. From \eqref{eq58} and \eqref{EQ5012}, we have
\begin{equation}\label{EQ5013}
\int_0^T \Vert \nabla_X v(\tau) \Vert_{L^2(\mathbb{R}^2)}^2 { \ } d\tau \leq C \Vert v_0 \Vert_{L^2(\mathbb{R}^2)}^2 + C \Vert Q \Vert_{L^2(0,T; L^2(\mathbb{R}^2))}^2.
\end{equation}
By \eqref{eq56} and \eqref{EQ5013}, we have
\begin{multline*}
\Vert dv/{dt} \Vert_{L^2(0,T; L^2(\mathbb{R}^2))}+ \Vert L v \Vert_{L^2(0,T; L^2(\mathbb{R}^2))}\\
\leq C(\kappa, \mathcal{M}) (\Vert v_0 \Vert_{W^{1,2} (\mathbb{R}^2) } + \Vert F \Vert_{L^2(0,T; L^2(\mathbb{R}^2))} + \Vert Q \Vert_{L^2(0,T; L^2(\mathbb{R}^2))}).
\end{multline*}
Since $C(\kappa , \mathcal{M})$ does not depend on $T$, we choose $T = \infty$ to have
\begin{multline*}
\Vert dv/{dt} \Vert_{L^2(0,\infty; L^2(\mathbb{R}^2))}+ \Vert L v \Vert_{L^2(0,\infty; L^2(\mathbb{R}^2))}\\
 \leq C(\kappa , \mathcal{M}) (\Vert v_0 \Vert_{W^{1,2} (\mathbb{R}^2)} + \Vert F \Vert_{L^2(0,\infty; L^2(\mathbb{R}^2))} + \Vert Q \Vert_{L^2(0,\infty; L^2(\mathbb{R}^2))}).
\end{multline*}
From Lemma \ref{lem24}, we see $(\rm{ii})$ if \eqref{eq52} holds. In the case when $F \equiv 0$, we have
\begin{equation}\label{EQ5014}
\Vert dv/{dt} \Vert_{L^2(0,\infty; L^2(\mathbb{R}^2))}+ \Vert L v \Vert_{L^2(0,\infty; L^2(\mathbb{R}^2))}\leq C(\kappa , \mathcal{M}) \Vert v_0 \Vert_{W^{1,2} (\mathbb{R}^2)}.
\end{equation}
From Lemma \ref{lem24}, we see $(\rm{iii})$ if \eqref{eq52} holds.

Let us now prove \eqref{eq52}. To this end, for $m \in \mathbb{N}$ we consider the following systems:
\begin{equation}\label{EQ5015}
\begin{cases}
\frac{d}{dt}v_1 + A v_1 = F \text{ on } (0,T_*),\\
v_1 \vert_{t =0} = v_0,
\end{cases}
\end{equation}
\begin{equation}\label{EQ5016}
\begin{cases}
\frac{d}{dt}v_{m+1} + A v_{m+1} = F - B v_m \text{ on } (0,T_*),\\
v_{m+1} \vert_{t =0} = v_0.
\end{cases}
\end{equation}
Here $T_* >0$ such that
\begin{equation*}
T_* \leq \min \{ 1, C_\star^2 \}.
\end{equation*}
Since $v_0 \in W^{1,2} (\mathbb{R}^2)$, $F \in L^2 (0,\infty;L^2(\mathbb{R}^2)) \cap C^\eta_{loc} ((0,\infty) ; L^2(\mathbb{R}^2))$, and $-A$ generates an analytic semigroup on $L^2(\mathbb{R}^2)$ from Lemma \ref{lem31}, we see that system \eqref{EQ5015} admits a unique strong solution $v_1$ in 
\begin{equation}\label{EQ5017}
C ([0,T_*);L^2 (\mathbb{R}^2) \cap C ((0,T_*);W^{2,2} (\mathbb{R}^2) \cap C^1 ((0,T_*);L^2 (\mathbb{R}^2)).
\end{equation}
Moreover, we see that $v_1$ is represented by
\begin{equation}\label{EQ5018}
v_1 (t) = {\rm{e}}^{-t A} v_0 + \int_0^t {\rm{e}}^{- ( t - \tau ) A} F ( \tau ) { \ }d \tau { \ }(t > 0).
\end{equation}
Since ${\rm{e}}^{- t A}$ is a contraction $C_0$-semigroup on $L^2(\mathbb{R}^2)$, we use \eqref{EQ5018} and the Cauchy-Schwarz inequality to check that for $t < T_*$
\begin{equation*}
\Vert v_1 (t) \Vert_{L^2(\mathbb{R}^2)} \leq \Vert v_0 \Vert_{L^2(\mathbb{R}^2)} + t^{1/2} \Vert F \Vert_{L^2(0, \infty; L^2(\mathbb{R}^2))} .
\end{equation*} 
From $T_* \leq 1$, we see that 
\begin{align}
\Vert v_1 \Vert_{L^\infty (0,T_*;L^2(\mathbb{R}^2))} & \leq \Vert v_0 \Vert_{L^2(\mathbb{R}^2)} + T_*^{1/2} \Vert F \Vert_{L^2(0, \infty; L^2(\mathbb{R}^2))}\notag\\
& \leq  \Vert v_0 \Vert_{L^2(\mathbb{R}^2)} + \Vert F \Vert_{L^2(0, \infty; L^2(\mathbb{R}^2))} < + \infty.\label{EQ5019}
\end{align} 
Here
\begin{equation*}
\Vert v_1 \Vert_{L^\infty (0,T_*;L^2(\mathbb{R}^2))} := \sup_{0 < t < T_*} \Vert v_1 (t) \Vert_{L^2(\mathbb{R}^2)}.
\end{equation*} 
Since $L^\infty(0,T_*;L^2(\mathbb{R}^2)) \subset L^2(0,T_*;L^2(\mathbb{R}^2))$, we find that
\begin{equation}\label{EQ5020}
v_1 \in L^2(0,T_*;L^2(\mathbb{R}^2)).
\end{equation}
From \eqref{eq32}, we see that
\begin{multline}\label{EQ5021}
\Vert d v_1/{dt} \Vert_{L^2(0,T_*;L^2(\mathbb{R}^2))} + \Vert A v_1 \Vert_{L^2(0,T_*;L^2(\mathbb{R}^2))}\\
\leq C_\star \Vert v_0  \Vert_{W^{1,2} (\mathbb{R}^2)} + C_\star \Vert F \Vert_{L^2(0,T_*;L^2(\mathbb{R}^2))}\\
\leq C_\star \Vert v_0  \Vert_{W^{1,2} (\mathbb{R}^2)} + C_\star \Vert F \Vert_{L^2(0,\infty;L^2(\mathbb{R}^2))} < + \infty.
\end{multline}
From \eqref{EQ5020}, \eqref{EQ5021}, \eqref{EQ4015}, and \eqref{eq33}, we see that
\begin{equation*}
B v_1 \in C_{loc}^{\eta} ((0,T_*);L^2(\mathbb{R}^2)) \cap L^2(0,T_*;L^2(\mathbb{R}^2)).
\end{equation*}
Since  $v_0 \in W^{1,2} (\mathbb{R}^2)$, $F - Bv_1 \in L^2 (0,T_*;L^2(\mathbb{R}^2)) \cap C^{ \eta }_{loc} ((0,T_*) ; L^2(\mathbb{R}^2))$, and $-A$ generates an analytic semigroup on $L^2(\mathbb{R}^2)$, we see that system \eqref{EQ5016} when $m=1$ admits a unique strong solution $v_2$ in \eqref{EQ5017}. Moreover, $v_2$ is represented by
\begin{equation*}
v_2 (t) = {\rm{e}}^{-t A} v_0 + \int_0^t {\rm{e}}^{- ( t - \tau ) A} F ( \tau ) { \ }d \tau - \int_0^t {\rm{e}}^{- ( t - \tau ) A} B v_1 ( \tau ) { \ }d \tau  { \ }(t > 0).
\end{equation*}
Since ${\rm{e}}^{- t A}$ is a contraction $C_0$-semigroup on $L^2(\mathbb{R}^2)$, we use the Cauchy-Schwarz inequality to check that for $t < T_*$
\begin{equation*}
\Vert v_2 (t) \Vert_{L^2(\mathbb{R}^2)} \leq \Vert v_0 \Vert_{L^2(\mathbb{R}^2)} + t^{1/2} \Vert F \Vert_{L^2(0, t; L^2(\mathbb{R}^2))}  + t^{1/2} \Vert Bv_m \Vert_{L^2(0, t; L^2(\mathbb{R}^2))}.
\end{equation*} 
Since $T_* \leq \min \{ 1, C_\star^2 \}$, we have
\begin{equation*}
\Vert v_2 \Vert_{L^\infty(0,T_*; L^2(\mathbb{R}^2) )} \leq \Vert v_0 \Vert_{L^2(\mathbb{R}^2)} + \Vert F \Vert_{L^2(0, \infty; L^2(\mathbb{R}^2))}  + C_\star \Vert Bv_1 \Vert_{L^2(0, T_*; L^2(\mathbb{R}^2))}.
\end{equation*} 
This implies that
\begin{equation}\label{EQ5022}
v_2 \in L^2(0,T_*; L^2(\mathbb{R}^2)).
\end{equation}
From \eqref{eq32}, we find that
\begin{multline}\label{EQ5023}
\Vert d v_2/{dt} \Vert_{L^2(0,T_*;L^2(\mathbb{R}^2))} + \Vert A v_2 \Vert_{L^2(0,T_*;L^2(\mathbb{R}^2))}\\
 \leq C_\star \Vert v_0  \Vert_{W^{1,2} (\mathbb{R}^2)} + C_\star \Vert F \Vert_{L^2(0,\infty;L^2(\mathbb{R}^2))} + C_\star \Vert B v_1 \Vert_{L^2(0,T_*;L^2(\mathbb{R}^2))} \\
 < + \infty .
\end{multline}
From \eqref{EQ5022}, \eqref{EQ5023}, \eqref{EQ4015} and \eqref{eq33}, we see that
\begin{equation*}
B v_2 \in C_{loc}^{ \eta } ((0,T_*);L^2(\mathbb{R}^2)) \cap L^2(0,T_*;L^2(\mathbb{R}^2)).
\end{equation*}
By induction, we see that for each $m \in \mathbb{N}$ system \eqref{EQ5016} admits a unique strong solution $v_{m+1}$ in \eqref{EQ5017}, satisfying that
\begin{equation}\label{EQ5024}
v_{m+1} (t) = {\rm{e}}^{-t A} v_0 + \int_0^t {\rm{e}}^{- ( t - \tau ) A} F ( \tau ) { \ }d \tau - \int_0^t {\rm{e}}^{- ( t - \tau ) A} B v_m ( \tau ) { \ }d \tau,
\end{equation}
\begin{multline}\label{EQ5025}
\Vert v_{m+1} \Vert_{L^\infty(0,T_*; L^2(\mathbb{R}^2))}\\
 \leq \Vert v_0 \Vert_{L^2(\mathbb{R}^2)} + \Vert F \Vert_{L^2(0, \infty; L^2(\mathbb{R}^2))}  + C_\star \Vert Bv_m \Vert_{L^2(0, T_*; L^2(\mathbb{R}^2))},
\end{multline} 
\begin{multline}\label{EQ5026}
\Vert d v_{m+1}/{dt} \Vert_{L^2(0,T_*;L^2(\mathbb{R}^2))} + \Vert A v_{m+1} \Vert_{L^2(0,T_*;L^2(\mathbb{R}^2))}\\
 \leq C_\star \Vert v_0  \Vert_{W^{1,2} (\mathbb{R}^2)} + C_\star \Vert F \Vert_{L^2(0,\infty;L^2(\mathbb{R}^2))} + C_\star \Vert B v_m \Vert_{L^2(0,T_*;L^2(\mathbb{R}^2))}\\ < + \infty ,
\end{multline}
$v_{m+1} \in L^2(0,T_* ; L^2(\mathbb{R}^2))$, and that
\begin{equation*}
B v_{m+1} \in C_{loc}^{ \eta } ((0,T_*);L^2(\mathbb{R}^2)) \cap L^2(0,T_*;L^2(\mathbb{R}^2)).
\end{equation*}

Set $X_T$ as follows:
\begin{equation*}
X_T = \{ \Phi \in L^2(0,T;L^2 (\mathbb{R}^2));{ \ } \Vert \Phi \Vert_{X_T} < \infty \},
\end{equation*}
where
\begin{equation*}
\Vert \Phi \Vert_{X_T} := \Vert \Phi \Vert_{L^\infty (0,T;L^2(\mathbb{R}^2))} +  \Vert d \Phi/{dt} \Vert_{L^2(0,T;L^2(\mathbb{R}^2))} + \Vert A \Phi \Vert_{L^2(0,T;L^2(\mathbb{R}^2))}.
\end{equation*}
Note that $L^\infty (0,T;L^2(\mathbb{R}^2)) \subset L^2(0,T; L^2(\mathbb{R}^2))$ since $T < \infty$. Now we prove that there are $T_{**} = T_{**}(\kappa, \mathcal{M} , \Vert \nabla_X \varphi \Vert_{L^\infty (\mathbb{R}^2)}) >0$ and $v_\infty \in X_{T_{**}}$ such that
\begin{equation*}
\lim_{m \to \infty} \Vert v_m - v_\infty \Vert_{X_{T_{**}}} = 0 .
\end{equation*}
By \eqref{EQ5019} and \eqref{EQ5021}, we have
\begin{equation}\label{EQ5027}
\Vert v_1 \Vert_{X_{T_*}} \leq (C_\star + 1) (\Vert v_0 \Vert_{W^{1,2}(\mathbb{R}^2)} +\Vert F \Vert_{L^2(0, \infty; L^2(\mathbb{R}^2))}).
\end{equation}
Now we prove that there is $T_{**} = T_{**} (\kappa, \mathcal{M} , \Vert \nabla_X \varphi \Vert_{L^\infty (\mathbb{R}^2)}) >0$ such that for each $m \in \mathbb{N}$
\begin{equation*}
\Vert v_m \Vert_{X_{T_{**}}} \leq 2 (C_\star + 1) (\Vert v_0 \Vert_{W^{1,2}(\mathbb{R}^2)} +\Vert F \Vert_{L^2(0, \infty; L^2(\mathbb{R}^2))}).
\end{equation*}
By \eqref{EQ5025} and \eqref{EQ5026}, we obtain
\begin{multline}\label{EQ5028}
\Vert v_{m+1} \Vert_{X_{T_*}} \leq (C_\star + 1) (\Vert v_0 \Vert_{L^2(\mathbb{R}^2)} + \Vert F \Vert_{L^2(0, \infty; L^2(\mathbb{R}^2))})\\
  + 2 C_\star \Vert Bv_m \Vert_{L^2(0, T_*; L^2(\mathbb{R}^2))}.
\end{multline} 
Using \eqref{eq47} and the interpolation theory, we check that
\begin{multline*}
\Vert B v_m \Vert_{L^2(\mathbb{R}^2)}\\
 \leq \frac{3}{2} \Vert \nabla_X \varphi \Vert_{L^\infty(\mathbb{R}^2)}^2 \Vert A v_m  \Vert_{L^2(\mathbb{R}^2)} + C_{\dagger}(\kappa, \mathcal{M}, \Vert \nabla_X \varphi \Vert_{L^\infty (\mathbb{R}^2)}) \Vert v_m  \Vert_{L^2(\mathbb{R}^2)}.
\end{multline*}
This gives
\begin{multline*}
2 C_\star \Vert B v_m \Vert_{L^2(0,T_*;L^2(\mathbb{R}^2))}\\
\leq 6 C_\star \Vert \nabla_X \varphi \Vert_{L^\infty}^2 \Vert A v_m  \Vert_{L^2(0,T_*;L^2(\mathbb{R}^2))} + 4 C_\dagger T_*^{1/2} \Vert v_m  \Vert_{L^\infty(0,T_*;L^2(\mathbb{R}^2))}.
\end{multline*}
We choose $T_{**} >0$ such that $4 C_\dagger T_{**}^{1/2} \leq 1/2$ and $T_{**} \leq T_*$. Since $\Vert \nabla_X \varphi \Vert_{L^\infty} \leq 1/(4\sqrt{C_\star})$, we see that $6 C_\star \Vert \nabla_X \varphi \Vert_{L^\infty}^2 \leq 1/2$. Thus, we have
\begin{equation}\label{EQ5029}
2 C_\star \Vert B v_m \Vert_{L^2(0,T_{**};L^2(\mathbb{R}^2))} \leq \frac{1}{2} \Vert v_m \Vert_{X_{T_{**}}}.
\end{equation}
From \eqref{EQ5028} and \eqref{EQ5029}, we find that for each $m \in \mathbb{N}$
\begin{equation}\label{EQ5030}
\Vert v_{m+1} \Vert_{X_{T_{**}}} \leq (C_\star + 1) (\Vert v_0 \Vert_{L^2(\mathbb{R}^2)} + \Vert F \Vert_{L^2(0, \infty; L^2(\mathbb{R}^2))})  + \frac{1}{2} \Vert v_m \Vert_{X_{T_{**}}}.
\end{equation} 
By \eqref{EQ5027} and \eqref{EQ5030}, we see that for each $m \in \mathbb{N}$
\begin{equation}\label{EQ5031}
\Vert v_m \Vert_{X_{T_{**}}} \leq 2 (C_\star + 1) (\Vert v_0 \Vert_{L^2(\mathbb{R}^2)} + \Vert F \Vert_{L^2(0, \infty; L^2(\mathbb{R}^2))})<+\infty.
\end{equation} 
Next, we prove that
\begin{equation*}
\lim_{m \to \infty} \Vert v_{m+1} - v_m \Vert_{X_{T_{**}}} =0.
\end{equation*}
It is clear that for $m \in \mathbb{N}$
\begin{equation*}
\begin{cases}
\frac{d}{dt}(v_{m+2}- v_{m+1}) + A (v_{m+2} - v_{m+1}) = - B (v_{m+1} - v_m) \text{ on } (0,T_{**}),\\
(v_{m+2} - v_{m+1}) \vert_{t =0} =0.
\end{cases}
\end{equation*}
From \eqref{eq32}, we have
\begin{multline}\label{EQ5032}
\Vert d (v_{m+2} - v_{m+1})/{dt} \Vert_{L^2(0,T_{**};L^2(\mathbb{R}^2))} + \Vert A (v_{m+2} - v_{m+1}) \Vert_{L^2(0,T_{**};L^2(\mathbb{R}^2))}\\
 \leq C_\star \Vert B (v_{m+1} - v_m) \Vert_{L^2(0,T_{**};L^2(\mathbb{R}^2))}.
\end{multline}
Since
\begin{equation*}
v_{m+2} (t) - v_{m+1} (t) = - \int_0^t {\rm{e}}^{- ( t - \tau ) A} B(v_{m+1}(\tau) - v_{m}(\tau)) { \ }d \tau
\end{equation*}
from \eqref{EQ5024}, we use the Cauchy-Schwarz inequality and $T_{**} \leq T_* \leq C_\star^2$ to see that
\begin{multline}\label{EQ5033}
\sup_{0 < t < T_{**}} \Vert v_{m+2} (t) - v_{m+1} (t) \Vert_{L^2(\mathbb{R}^2)} \leq T_{**}^{1/2} \Vert B(v_{m+1} - v_{m}) \Vert_{L^2(0,T_{**};L^2(\mathbb{R}^2))}\\
\leq C_\star \Vert B(v_{m+1} - v_{m}) \Vert_{L^2(0,T_{**};L^2(\mathbb{R}^2))}.
\end{multline}
By \eqref{EQ5032} and \eqref{EQ5033}, we have
\begin{equation*}
\Vert v_{m+2} - v_{m+1} \Vert_{X_{T_{**}}} \leq 2C_\star \Vert B (v_{m+1} - v_m) \Vert_{X_{T_{**}}}.
\end{equation*}
By the previous argument to derive \eqref{EQ5029}, we find that for each $m \in \mathbb{N}$
\begin{align*}
\Vert v_{m+2} - v_{m+1} \Vert_{X_{T_{**}}} \leq \frac{1}{2} \Vert v_{m+1} - v_m \Vert_{X_{T_{**}}}.
\end{align*}
This implies that
\begin{align*}
\Vert v_{m+2} - v_{m+1} \Vert_{X_{T_{**}}} & \leq \frac{1}{2^m}  \Vert v_2 - v_1 \Vert_{X_{T_{**}}}\\
 & \leq \frac{1}{2^m}  ( \Vert v_2 \Vert_{X_{T_{**}}} + \Vert v_1 \Vert_{X_{T_{**}}}).
\end{align*}
By \eqref{EQ5031}, we check that
\begin{multline}\label{EQ5034}
\Vert v_{m+2} - v_{m+1} \Vert_{X_{T_{**}}} \leq \frac{4(C_\star +1)}{2^m}  ( \Vert v_0 \Vert_{W^{1,2} (\mathbb{R}^2)} + \Vert F \Vert_{L^2(0,\infty; L^2(\mathbb{R}^2))}  )\\
 \to 0 \text{ as }m \to \infty.
\end{multline}

From a fixed-point argument, there exists a unique function $v_\infty \in X_{T_{**}}$ such that
\begin{equation}\label{EQ5035}
\lim_{m \to \infty} \Vert v_m - v_\infty \Vert_{X_{T_{**}}} = 0 .
\end{equation}
Applying \eqref{EQ5035} into system \eqref{EQ5016} and \eqref{EQ5024}, we find that
\begin{equation*}
dv_\infty/{dt} + Av_\infty = F - B v_\infty \text{ on }(0,T_{**}),
\end{equation*}
and
\begin{equation*}
v_\infty (t) = {\rm{e}}^{-t A} v_0 + \int_0^t {\rm{e}}^{- ( t - \tau ) A} F ( \tau ) { \ }d \tau - \int_0^t {\rm{e}}^{- ( t - \tau ) A} B v_\infty ( \tau ) { \ }d \tau{ \ }(0<t<T_{**}) .
\end{equation*}
Using the Cauchy-Schwarz inequality, we check that
\begin{multline*}
\Vert v_\infty (t) - v_0 \Vert_{L^2(\mathbb{R}^2)}\\ = \Vert {\rm{e}}^{-t A} v_0 -v_0 \Vert_{L^2(\mathbb{R}^2)} + t^{1/2} \Vert F \Vert_{L^2(0,\infty;L^2(\mathbb{R}^2))} + t^{1/2} \Vert Bv_\infty \Vert_{L^2(0,T_{**};L^2(\mathbb{R}^2))}.
\end{multline*}
Since ${\rm{e}}^{-t A}$ is a $C_0$-semigroup on $L^2(\mathbb{R}^2)$, we see that
\begin{equation*}
\lim_{t \to 0 +0} \Vert v_\infty (t) - v_0 \Vert_{L^2(\mathbb{R}^2)} = 0.
\end{equation*}
That is, $v_\infty$ satisfies
\begin{equation}\label{EQ5036}
\begin{cases}
dv_\infty/{dt} + Lv_\infty = F \text{ on }(0,T_{**}),\\
v_\infty\vert_{t =0} = v_0,
\end{cases}
\end{equation}
From the uniqueness of the solutions to system \eqref{EQ5036}, we see that $v= v_\infty$ on $[0,T_{**})$.

Now we consider the two cases when $T \leq T_{**}$ and $T > T_{**}$. We first consider the case when $T \leq T_{**}$. It is clear that
\begin{equation*}
\Vert B v \Vert_{L^2(0,T;L^2(\mathbb{R}^2))} = \Vert B v_\infty \Vert_{L^2(0,T;L^2(\mathbb{R}^2))} < + \infty.
\end{equation*}
Next, we consider the case when $T > T_{**}$. Since $v \in C([T_{**}, T] ; W^{2,2} (\mathbb{R}^2) )$, we use Lemma \ref{lem23} to see that
\begin{multline*}
\Vert B v \Vert_{L^2(0,T;L^2(\mathbb{R}^2))} = \Vert B v_\infty \Vert_{L^2(0,T_{**};L^2(\mathbb{R}^2))} + \Vert B v \Vert_{L^2(T_{**},T;L^2(\mathbb{R}^2))}\\
\leq \Vert B v_\infty \Vert_{L^2(0,T_{**};L^2(\mathbb{R}^2))} + C T^{1/2} \Vert v \Vert_{L^\infty (T_{**},T;W^{2,2}(\mathbb{R}^2))} < + \infty.
\end{multline*}
Thus, we see \eqref{eq52}.

Finally, we show $(\rm{iv})$. Assume that $F \in L^2(\mathbb{R}^2 \times (0,\infty))$. Let us recall the heat kernels $\mathcal{E}[t] = \mathcal{E} (X ,t) = {\rm{exp}}(- \vert X \vert^2/{4\kappa t})/{ 4 \pi \kappa t}$ for $\mathbb{R}^2$. Since ${\rm{e}}^{- tA} f = \mathcal{E}[t]*f$ for $f \in L^2(\mathbb{R}^2)$, we consider \eqref{EQ5018} and \eqref{EQ5024} as
\begin{align*}
v_1 (t) & =  \mathcal{E}[t] *v_0 + \int_0^t \mathcal{E}[t-\tau]* F ( \tau ) { \ }d \tau,\\
v_{m+1} (t) & = \mathcal{E}[t]* v_0 + \int_0^t \mathcal{E}[t-\tau]* F ( \tau ) { \ }d \tau - \int_0^t \mathcal{E}[t-\tau]* B v_m ( \tau ) { \ }d \tau .
\end{align*}
We easily check that for each $m \in \mathbb{N}$, $v_m \in L^2( \mathbb{R}^2 \times (0,T))$. By \eqref{EQ5034} and $v_m \in L^2( 0,T; L^2(\mathbb{R}^2))$, we see that
\begin{equation*}
\Vert v_{m+1} - v_m \Vert_{L^2( \mathbb{R}^2 \times (0,T_{**}) )} = \Vert v_{m+1} - v_m \Vert_{L^2(0,T_{**};L^2(\mathbb{R}^2) )} \to 0  \text{ as }m \to \infty.
\end{equation*}
This shows that there is $\tilde{v} \in L^2( \mathbb{R}^2 \times (0,T))$ such that
\begin{equation*}
\Vert v_{m} - \tilde{v} \Vert_{L^2( \mathbb{R}^2 \times (0,T_{**}) )}  \to 0  \text{ as }m \to \infty.
\end{equation*}
By \eqref{EQ5035}, we see that $\tilde{v} = v_\infty =v$ on $(0,T_{**})$. Since $T_{**}$ does not depend on the initial datum, we  repeat this argument for each interval $(mT_{**}, (m+1)T_{**})$ $(m \in \mathbb{N})$ to find that $v \in L^2(\mathbb{R}^2 \times (0, T ))$ for each fixed $T \in (0,\infty)$. Therefore, Theorem \ref{thm26} is proved.

  \end{proof}

%\newpage

\section{Weighted Laplace-Beltrami Operator}\label{sect6}

Let us study our weighted Laplace-Beltrami operator. Let $\kappa >0$ and $\varphi \in BC^2(\mathbb{R}^2)$. Let $L$, $A$, and $B$, be the three operators defined by \eqref{eq25}, \eqref{eq31}, and \eqref{eq46}, respectively.

\begin{proposition}\label{prop61}
Set the operator $\mathcal{L}$ on $L^2(\mathbb{R}^2)$ as follows:
\begin{equation}\label{eq61}
\begin{cases}
{\displaystyle{\mathcal{L} f = \sqrt{G} L f =- \kappa \frac{\partial}{\partial X_\alpha} \bigg( \sqrt{G} g^{\alpha \beta} \frac{\partial f }{\partial X_\beta} \bigg),}}\\
D ( \mathcal{L}) = W^{2,2} (\mathbb{R}^2)(=D(A)).
\end{cases}
\end{equation}
Assume that
\begin{equation*}
\Vert \nabla_X \varphi \Vert_{L^\infty (\mathbb{R}^2)} \leq \frac{1}{4}.
\end{equation*}
Then the four assertions hold:\\
\noindent $(\rm{i})$ The operator $-\mathcal{L}$ generates a $C_0$-semigroup on $L^2(\mathbb{R}^2)$.\\
\noindent $(\rm{ii})$ The operator $-\mathcal{L}$ generates an analytic semigroup on $L^2(\mathbb{R}^2)$.\\
\noindent $(\rm{iii})$ The operator $\mathcal{L}$ is a selfadjoint operator on $L^2(\mathbb{R}^2)$.\\
\noindent $(\rm{iv})$ $R (\mathcal{L})$ is dense in $L^2(\mathbb{R}^2)$, where $R(\mathcal{L})$ is the range of $\mathcal{L}$. 
\end{proposition}

\begin{proof}[Proof of Proposition \ref{prop61}]
Assume that $\Vert \nabla_X \varphi \Vert_{L^\infty (\mathbb{R}^2)} \leq 1/4$. We first show $(\rm{i})$ and $(\rm{ii})$. Using $\mathcal{L} = \sqrt{G}L$, $L = A + B$, and \eqref{EQ4015}, we check that for all $f \in D (A)$
\begin{align*}
\Vert \mathcal{L} f - A f \Vert_{L^2(\mathbb{R}^2)} & =\Vert \mathcal{L} f - L f + Lf - A f \Vert_{L^2(\mathbb{R}^2)}\\
& \leq \Vert (\sqrt{G} - 1) L f \Vert_{L^2(\mathbb{R}^2)} + \Vert B f \Vert_{L^2(\mathbb{R}^2)}\\
& \leq \Vert \nabla_X \varphi \Vert_{L^\infty(\mathbb{R}^2)}^2 \Vert (A+B) f \Vert_{L^2(\mathbb{R}^2)} + \Vert B f \Vert_{L^2(\mathbb{R}^2)}\\
& \leq \frac{21}{64}\Vert A f \Vert_{L^2(\mathbb{R}^2)} + C(\kappa , \Vert \nabla_X^2 \varphi \Vert_{L^\infty (\mathbb{R}^2)})  \Vert f \Vert_{L^2(\mathbb{R}^2)}.
\end{align*}
Since $- A$ generates an analytic contraction semigroup on $L^2 (\mathbb{R}^2)$, it follows from the perturbation theory on the semigroup theory(\cite[Chapter 3]{Paz83}, \cite[Chapter III]{EN00}) to see that $- \mathcal{L}$ generates an analytic $C_0$-semigroup on $L^2(\mathbb{R}^2)$. Thus, we see $(\rm{i})$ and $(\rm{ii})$.

Next, we prove $(\rm{iii})$. Since $D (\mathcal{L})$ is dense in $L^2(\mathbb{R}^2)$, we can define the adjoint operator of $\mathcal{L}$. Let $\mathcal{L}^*$ be the adjoint operator of $\mathcal{L}$. Define the operator $\mathcal{L}'$ on $L^2(\mathbb{R}^2)$ as follows: 
\begin{equation*}
\begin{cases}
{\displaystyle{\mathcal{L}' f = - \kappa \frac{\partial}{\partial X_\alpha} \bigg( \sqrt{G} g^{\alpha \beta} \frac{\partial f }{\partial X_\beta} \bigg),}}\\
D ( \mathcal{L}') = W^{2,2} (\mathbb{R}^2).
\end{cases}
\end{equation*}
By integration by parts, we check that for all $f, \phi \in D (\mathcal{L})( = W^{2,2} (\mathbb{R}^2))$
\begin{equation*}
\left< \mathcal{L} f ,\phi \right>_{L^2(\mathbb{R}^2)} =\left< f, \mathcal{L}' \phi \right>_{L^2(\mathbb{R}^2)} =\left< f, \mathcal{L}^* \phi \right>_{L^2(\mathbb{R}^2)}.
\end{equation*}
This shows that $\mathcal{L}^* = \mathcal{L}'$ on $W^{2,2} (\mathbb{R}^2)$. Since $- \mathcal{L}' = - \mathcal{L}$, we see that $- \mathcal{L}'$ generates an analytic semigroup on $L^2(\mathbb{R}^2)$. Therefore, there is $\lambda_0 >0$ such that $\lambda_0 \in \rho (- \mathcal{L}')$, where $\rho (-\mathcal{L}')$ is the resolvent set of $- \mathcal{L}'$.

Now we show that $D(\mathcal{L}^*) = D (\mathcal{L}')$ and $\mathcal{L}^* = \mathcal{L}'$. To this end, we show that the operator $(\mathcal{L}^* + \lambda_0) $ is injective. Let $\phi_0 \in D (\mathcal{L}^*)$ such that $(\mathcal{L}^* + \lambda_0 ) \phi_0 =0$. By the definition of the adjoint operator, we check that for all $f \in D(\mathcal{L})$
\begin{equation*}
0= \left< f , (\mathcal{L}^* + \lambda_0 ) \phi_0 \right>_{L^2(\mathbb{R}^2)} =\left< (\mathcal{L} + \lambda_0 ) f, \phi_0 \right>_{L^2(\mathbb{R}^2)} .
\end{equation*}
Since $\lambda_0 \in \rho (- \mathcal{L}') = \rho (- \mathcal{L})$, we find that $\phi_0 = 0$. This shows that $(\mathcal{L}^* + \lambda_0 )$ is injective. Since we have already known that $\mathcal{L}' \subset \mathcal{L}^*$, we now prove that $\mathcal{L}^* \subset \mathcal{L}'$. Fix $\phi_1 \in D(\mathcal{L}^*)$. By the definition of the domain $D(\mathcal{L}^*)$, there is $h_0 \in L^2(\mathbb{R}^2)$ such that
\begin{equation*}
(\mathcal{L}^* + \lambda_0) \phi_1 = h_0.
\end{equation*}
Since $h_0 \in L^2(\mathbb{R}^2)$ and $\lambda_0 \in \rho (- \mathcal{L}')$, there is $\phi_2 \in D(\mathcal{L}')(=W^{2,2} (\mathbb{R}^2))$ such that
\begin{equation*}
(\mathcal{L}' + \lambda_0 ) \phi_2 = h_0. 
\end{equation*}
From $\mathcal{L}' = \mathcal{L}^*$ on $W^{2,2} (\mathbb{R}^2)$, we find that
\begin{equation*}
0 = ( \mathcal{L}^* + \lambda_0) \phi_2 - (\mathcal{L}' + \lambda_0) \phi_1 = (\mathcal{L}^* + \lambda_0) (\phi_2 - \phi_1 ).
\end{equation*}
Since $(\mathcal{L}^* + \lambda_0)$ is injective, we see that $\phi_2 = \phi_1$ and $\phi_2 \in D(\mathcal{L}') = W^{2,2} (\mathbb{R}^2)$. Therefore, we conclude that $D (\mathcal{L}^*)\subset D (\mathcal{L}') = W^{2,2} (\mathbb{R}^2)$. Since $\mathcal{L}^* = \mathcal{L}' = \mathcal{L}$ and $D (\mathcal{L}^*) = D(\mathcal{L}') = D(\mathcal{L}) = W^{2,2} (\mathbb{R}^2)$, we see $(\rm{iii})$.

Finally, we prove $(\rm{iv})$. Since $\mathcal{L}$ is a closed operator, we show that $\mathcal{N} (\mathcal{L}^*) = \{ 0 \}$, where $\mathcal{N}(\mathcal{L}^*)$ is the null set of $\mathcal{L}^*$. Let $\phi_* \in D (\mathcal{L})$ such that
\begin{equation*}
0 = \left< \phi_* , \mathcal{L}^* \phi_* \right>_{L^2(\mathbb{R}^2)}.
\end{equation*}
By \eqref{eq35} and \eqref{eq44}, we check that
\begin{align*}
0 = \left< \phi_* , \mathcal{L}^* \phi_* \right>_{L^2(\mathbb{R}^2)} & = \kappa \int_{\mathbb{R}^2} g^{\alpha \beta} \frac{\partial \phi_*}{\partial X_\alpha} \frac{\partial \phi_*}{\partial X_\beta} \sqrt{G} { \ }dX\\
&\geq \frac{\kappa}{\sqrt{3}} \Vert \nabla_X \phi_* \Vert_{L^2(\mathbb{R}^2)}^2 \geq 0.
\end{align*}
This shows that $\phi_* =0$, that is, $\mathcal{N} (\mathcal{L}^*) = \{ 0 \}$. Since $\mathcal{L}$ is a closed operator and $D (\mathcal{L})$ is dense in $L^2(\mathbb{R}^2)$, we see that $\overline{\mathcal{R}(\mathcal{L})} = \mathcal{N}(\mathcal{L}^*)^\bot = L^2(\mathbb{R}^2)$. Thus, we see $(\rm{iv})$. Therefore, Proposition \ref{prop61} is proved.

  \end{proof}

%\newpage

\section{Stability}\label{sect7}

We derive asymptotic $L^2$-stability of solutions to our system by applying both maximal $L^2$-regularity of the Laplace-Beltrami operator $L$ and nice properties of the weighted Laplace-Beltrami operator $\mathcal{L}$. Let $\kappa >0$, $\varphi \in BC^2(\mathbb{R}^2)$, and let $C_\star = C_\star (\kappa ) >0$ be the constant appearing in \eqref{eq32}. Let $L$ and $\mathcal{L}$ be the two operators defined by \eqref{eq25} and \eqref{eq61}, respectively.

\begin{proof}[Proof of Theorem \ref{thm27}]
Assume that
\begin{equation*}
\Vert \nabla_X \varphi \Vert_{L^\infty (\mathbb{R}^2)} \leq \min \left\{ \frac{1}{4} , \frac{1} {4 \sqrt{C_\star}} \right\}.
\end{equation*}
Fix $w_0 \in L^2 (\mathbb{R}^2)$. Set $w (t) = {\rm{e}}^{- t L} w_0 { \ }(t>0)$. Since $- L$ generates an analytic semigroup on $L^2(\mathbb{R}^2)$, we see that $w$ is a unique strong solution to system
\begin{equation*}
\begin{cases}
dw/{dt} + L w = 0 \text{ on }(0,\infty),\\
w\vert_{t =0} = w_0.
\end{cases}
\end{equation*}

Now we prove that
\begin{equation}\label{eq71}
\lim_{t \to \infty} \Vert {\rm{e}}^{ - t L } w_0 \Vert_{L^2( \mathbb{R}^2)} = 0.
\end{equation}
Let $\varepsilon >0$. Since $\sqrt{G} w_0 \in L^2(\mathbb{R}^2)$ and $R (\mathcal{L})$ is dense in $L^2(\mathbb{R}^2)$ from Proposition \ref{prop61}, we take $W_0 \in W^{2,2} (\mathbb{R}^2)$ such that
\begin{equation}\label{eq72}
\Vert \sqrt{G} w_0 - \mathcal{L} W_0 \Vert_{L^2 (\mathbb{R}^2)} < \frac{\varepsilon}{6}.
\end{equation}
Since $\mathcal{L} = \sqrt{G} L$ on $W^{2,2}(\mathbb{R}^2)$, we use \eqref{eq72} and \eqref{eq43} to see that
\begin{align}
\frac{\varepsilon}{6} &> \Vert \sqrt{G} w_0 - \sqrt{G} L W_0 \Vert_{L^2 (\mathbb{R}^2)}\notag\\
& \geq \Vert w_0 - L W_0 \Vert_{L^2 (\mathbb{R}^2)} .\label{eq73}
\end{align}
By \eqref{eq41} and \eqref{eq73}, we check that for $t >0$
\begin{align}
\Vert {\rm{e}}^{ - t L} w_0 \Vert_{L^2(\mathbb{R}^2)} & \leq \Vert {\rm{e}}^{ - t L} (w_0 - L W_0 ) \Vert_{L^2(\mathbb{R}^2)} + \Vert {\rm{e}}^{ - t L} L W_0 \Vert_{L^2(\mathbb{R}^2)}\notag\\
& \leq 3^{1/4} \Vert w_0 - L W_0 \Vert_{L^2(\mathbb{R}^2)} + \Vert {\rm{e}}^{ - t L} L W_0 \Vert_{L^2(\mathbb{R}^2)}\notag\\
& \leq  \varepsilon/2 + \Vert {\rm{e}}^{ - t L} L W_0 \Vert_{L^2(\mathbb{R}^2)}.\label{eq74}
\end{align}
From \eqref{eq41}, we also see that for all $0 \leq \tau <t$,
\begin{equation*}
\Vert {\rm{e}}^{ - t L} L W_0 \Vert_{L^2(\mathbb{R}^2)} \leq 3^{1/4}  \Vert {\rm{e}}^{ - \tau L} L W_0 \Vert_{L^2(\mathbb{R}^2)}.
\end{equation*}
Integrating with respect to $\tau$, and then using the Cauchy-Schwarz inequality, we observe that
\begin{align}
\Vert {\rm{e}}^{ - t L} L W_0 \Vert_{L^2(\mathbb{R}^2)} & \leq \frac{3^{1/4}}{t} \int_0^t \Vert {\rm{e}}^{ - \tau L} L W_0 \Vert_{L^2(\mathbb{R}^2)} { \ }d \tau\notag\\
& \leq \frac{3^{1/4}}{t^{1/2}} \bigg( \int_0^t \Vert {\rm{e}}^{ - \tau L} L W_0 \Vert_{L^2(\mathbb{R}^2)}^2 { \ }d \tau \bigg)^{1/2}\notag\\
& = \frac{3^{1/4}}{t^{1/2}} \bigg( \int_0^t \Vert L {\rm{e}}^{ - \tau L} W_0 \Vert_{L^2(\mathbb{R}^2)}^2 { \ }d \tau \bigg)^{1/2}.\label{eq75}
\end{align}
Now we set $W(t) = {\rm{e}}^{- t L} W_0$. It is clear that $W$ satisfies
\begin{equation*}
\begin{cases}
dW/{dt}  + L W= 0 \text{ on }(0,\infty),\\
W\vert_{t=0} = W_0.
\end{cases}
\end{equation*}
From \eqref{EQ5014} and Lemma \ref{lem24}, we have
\begin{multline}\label{eq76}
\Vert dW/{dt} \Vert_{L^2(0,\infty; L^2(\mathbb{R}^2))}+ \Vert L W \Vert_{L^2(0,\infty; L^2(\mathbb{R}^2))}\\
 \leq C(\kappa, \Vert \nabla_X^2 \varphi \Vert_{L^\infty (\mathbb{R}^2)}) \Vert W_0 \Vert_{W^{1,2}(\mathbb{R}^2)} .
\end{multline}
Since
\begin{equation*}
\Vert L W \Vert_{L^2(0,\infty; L^2(\mathbb{R}^2))} = \bigg( \int_0^\infty \Vert L {\rm{e}}^{- \tau L} W_0 \Vert_{L^2(\mathbb{R}^2)}^2 { \ }d \tau \bigg)^{1/2},
\end{equation*}
we use \eqref{eq73}-\eqref{eq76} to check that
\begin{align*}
\Vert {\rm{e}}^{ - t L} w_0 \Vert_{L^2(\mathbb{R}^2)} & \leq \Vert {\rm{e}}^{ - t L} (w_0 - L W_0 ) \Vert_{L^2(\mathbb{R}^2)} + \Vert {\rm{e}}^{ - t L} L W_0 \Vert_{L^2(\mathbb{R}^2)}\\
& \leq \frac{1}{2} \varepsilon + \frac{C(\kappa , \Vert \nabla_X^2 \varphi \Vert_{L^\infty (\mathbb{R}^2)})}{t^{1/2}}\Vert W_0 \Vert_{W^{1,2}(\mathbb{R}^2)} \\
& < \varepsilon { \ }(t>>1).
\end{align*}
Since $\varepsilon$ is arbitrary, we see \eqref{eq71}. From \eqref{eq43}, we find that
\begin{equation*}
\lim_{t \to \infty} \Vert \sqrt{G} {\rm{e}}^{ - t L } w_0 \Vert_{L^2( \mathbb{R}^2)} = 0.
\end{equation*}
Therefore, Theorem \ref{thm27} is proved.
  \end{proof}

%\newpage

\section{Appendix: Function Spaces on Surfaces and H\"{o}lder Continuity}\label{sect8}

In Appendix, we first characterize the function spaces $L^2(\Gamma \times (0,T) )$ and $L^2(\Gamma \times (0,\infty))$, and then we prepare one tool to derive the H\"{o}lder continuity of the solutions to our system.

\subsection{Function Spaces on Surfaces}\label{subsec81}
Let $\varphi \in BC^2(\mathbb{R}^2)$, and $\mathcal{T} \in (0,\infty]$. Set
\begin{equation*}
[0, \mathcal{T}] = \begin{cases}
[0,\mathcal{T}] \text{ if }\mathcal{T} < \infty,\\
[0,\infty ) \text{ if } \mathcal{T}=\infty.
\end{cases} 
\end{equation*}
Define
\begin{align*}
L^2(\Gamma \times (0,\mathcal{T})) & = \overline{C_0( \Gamma \times [0,\mathcal{T}])}^{\Vert \cdot \Vert_{L^2(\Gamma \times (0,\mathcal{T}))} },\\
C_0 (\Gamma \times [0, \mathcal{T}] ) & = \{\mathcal{F} :\Gamma \times [0,\mathcal{T}] \to \mathbb{R};  \widehat{\mathcal{F}} \in C_0(\mathbb{R}^2 \times [0, \mathcal{T}] ) \}. 
\end{align*}
Here $\widehat{\mathcal{F}}= \widehat{\mathcal{F}}(X,t) := \mathcal{F}(X_1,X_2,\varphi (X_1,X_2),t)$ for almost all $(X ,t) \in \mathbb{R}^2 \times [0, \mathcal{T}]$ and
\begin{equation*}
\Vert \mathcal{F} \Vert_{L^2(\Gamma \times (0,\mathcal{T}))} := \Vert G^{1/4} \widehat{\mathcal{F}} \Vert_{L^2(\mathbb{R}^2 \times (0, \mathcal{T}))}.
\end{equation*}

\begin{lemma}\label{lem81}
Let $f \in L^2(\mathbb{R}^2 \times (0,\mathcal{T}))$. For almost all $(x,t) \in \Gamma \times (0,\mathcal{T})$, we set
\begin{equation*}
\tilde{f} = \tilde{f}(x,t) = f(x_1,x_2,t).
\end{equation*}
Then $\tilde{f} \in L^2(\Gamma \times (0,\mathcal{T}))$.
\end{lemma}
\noindent Remark that for almost all $(x,t) \in \Gamma \times (0,\mathcal{T})$ means that for almost all $(x_1,x_2,t) \in \mathbb{R}^2 \times (0,\mathcal{T})$ since $x_3= \varphi(x_1,x_2)$.  
\begin{proof}[Proof of Lemma \ref{lem81}]
Fix $f \in L^2(\mathbb{R}^2 \times (0,\mathcal{T}))$. Since $C_0(\mathbb{R}^2 \times [0,\mathcal{T}])$ is dense in $L^2(\mathbb{R}^2 \times (0,\mathcal{T}))$, there is $\{ f_m \} \subset L^2(\mathbb{R}^2 \times (0,\mathcal{T}))$ such that
\begin{equation}\label{eq81}
\lim_{m \to \infty} \Vert f - f_m \Vert_{L^2(\mathbb{R}^2 \times (0,\mathcal{T}))} =0.
\end{equation}
For every ${ }^t(x_1,x_2,x_3)\in \Gamma$ and $t \in [0,\mathcal{T}]$, we set
\begin{equation*}
\tilde{f}_m = \tilde{f}_m(x_1,x_2,x_3,t) = f_m  (x_1,x_2,t).
\end{equation*}
By the definition of $\Vert \cdot \Vert_{L^2(\Gamma \times (0, \mathcal{T}))}$ and \eqref{eq81}, we check that
\begin{align*}
\Vert \tilde{f} - \tilde{f}_m \Vert_{L^2(\Gamma \times (0,\mathcal{T}))} & = \Vert G^{1/4}(f - f_m) \Vert_{L^2(\mathbb{R}^2 \times (0,\mathcal{T}))}\\
& \leq C (\Vert \nabla_X \varphi \Vert_{L^\infty (\mathbb{R}^2)}) \Vert f - f_m \Vert_{L^2(\mathbb{R}^2 \times (0,\mathcal{T}))} \to 0 { \ }(m \to \infty)
\end{align*}
Therefore, we see that $\tilde{f} \in L^2(\Gamma \times (0,\mathcal{T}))$.
  \end{proof}

\subsection{H\"{o}lder Continuity}\label{subsec82}

\begin{lemma}\label{lem82}
Let $X$ be a Banach space, and $\Vert \cdot \Vert_X$ be its norm. Let $\mathcal{A}$ be a linear operator densely defined in $X$. Let $V_0 \in X$ and $F \in L^p(0,\mathcal{T};X) \cap C_{loc}^\eta ((0,\mathcal{T});X)$ for some $1 \leq p < \infty$, $0< \eta \leq 1$, and $\mathcal{T} \in (0,\infty]$. Assume that $- \mathcal{A}$ generates an analytic semigroup on $X$. Then system
\begin{equation*}
\begin{cases}
dV/{dt} + \mathcal{A} V = F { \ } \text{ on }(0,\mathcal{T}),\\
V\vert_{t = 0} = V_0,
\end{cases}
\end{equation*}
admits a unique strong solution $V$ in
\begin{equation*}
C([0,\mathcal{T}) ; X ) \cap C((0,\mathcal{T}) ; D (\mathcal{A})) \cap C^1((0,\mathcal{T}) ; X).
\end{equation*}
Moreover, $V$ satisfies that $V$ is represented by
\begin{equation*}
V (t) = {\rm{e}}^{-t \mathcal{A}} V_0 + \int_0^t {\rm{e}}^{ - (t- \tau) \mathcal{A} } F (\tau) { \ }d \tau { \ }(0<t<\mathcal{T}),
\end{equation*} 
and that
\begin{equation}\label{eq82}
V , \mathcal{A} V , dV/{dt} \in C_{loc}^{ \eta } ((0,\mathcal{T}); X).
\end{equation}
\end{lemma}
Using the same arguments as in \cite[Section 4.3]{Paz83} and \cite[Section 4.3]{Lun95}, we can obtain Lemma \ref{lem82}. For the readers, we give a sketch of the proof.
\begin{proof}[Proof of Lemma \ref{lem82}]
We only prove \eqref{eq82}. Since $- \mathcal{A}$ generates an analytic semigroup on $X$, we see that there is $\lambda_0 \in \rho (- \mathcal{A})$ such that $\lambda_0 >0$, where $\rho (- \mathcal{A})$ is the resolvent set of $-\mathcal{A}$.
  We also see that there is $C = C(\lambda_0) >0$ such that for all $f \in D (\mathcal{A})$, $f_* \in X$, and $t >0$
\begin{align}
\Vert f \Vert_{X} + \Vert \mathcal{A} f \Vert_X & \leq C \Vert (\mathcal{A} + \lambda_0 ) f \Vert_X ,\label{eq83}\\
\Vert {\rm{e}}^{- t \mathcal{A}} f_* \Vert_{X} & \leq C {\rm{e}}^{t \lambda_0} \Vert f_* \Vert_X .\label{eq84}
\end{align}
Therefore, we can define the fractional powers of $(\mathcal{A} + \lambda_0)$.

Since $- ( \mathcal{A} + \lambda_0)$ generates a bounded analytic semigroup on $X$, it follows from \cite[Theorem 6.13 in Chapter II]{Paz83} to find that for each $0<q \leq 1$ and $\mathcal{T}_*>0$ there is $C (\mathcal{T}_*) = C (q,\mathcal{T}_*, \lambda_0) >0$ such that for all $f_1 \in X$, $f_2 \in D ((\mathcal{A} + \lambda_0)^q)$, and $0 < t \leq \mathcal{T}_*$,
\begin{align}
\Vert (\mathcal{A} + \lambda_0 )^q {\rm{e}}^{- t (\mathcal{A}+ \lambda_0) } f_1 \Vert_{X} & \leq C(\mathcal{T}_*) t^{-q} \Vert f_1 \Vert_{X},\label{eq85}\\
\Vert ({\rm{e}}^{- t (\mathcal{A} + \lambda_0)} -  1) f_2 \Vert_{X} & \leq C(\mathcal{T}_*)  t^q \Vert (\mathcal{A} + \lambda_0)^q f_2 \Vert_{X}.\label{eq86}
\end{align}
Let $\mathcal{T}_*>0$ and $0 < q \leq 1$. By \eqref{eq84} and \eqref{eq85}, we check that for each $f_1 \in X$ and $0<t \leq \mathcal{T}_*$
\begin{align*}
\Vert (\mathcal{A} + \lambda_0 )^q {\rm{e}}^{- t \mathcal{A} } f_1 \Vert_{X} & = \Vert (\mathcal{A} + \lambda_0 )^q {\rm{e}}^{t\lambda_0}{\rm{e}}^{- t (\mathcal{A} + \lambda_0) } f_1 \Vert_{X}\\
 & \leq C( q, \mathcal{T}_*, \lambda_0 ) t^{-q} \Vert f_1 \Vert_{X} .
\end{align*}
From \eqref{eq83}, we also find that
\begin{align*}
\Vert \mathcal{A} {\rm{e}}^{- t \mathcal{A} } f_1 \Vert_{X} & = C \Vert (\mathcal{A} + \lambda_0) {\rm{e}}^{- t \mathcal{A} } f_1 \Vert_{X}\\
 & \leq C( q, \mathcal{T}_*, \lambda_0 ) t^{-1} \Vert f_1 \Vert_{X}.
\end{align*}
Using \eqref{eq84}, \eqref{eq86}, and the Taylor expansion of ${\rm{e}}^{- \lambda_0 t}$, we check that for each $f_2 \in D ( ( \mathcal{A} + \lambda_0)^q )$ and $0<t \leq \mathcal{T}_*$
\begin{align*}
\Vert ({\rm{e}}^{- t \mathcal{A}} -  1) f_2 \Vert_{X} & \leq \Vert ({\rm{e}}^{- t \mathcal{A}} -  {\rm{e}}^{- t( \mathcal{A} + \lambda_0)} ) f_2 \Vert_{X} + \Vert ({\rm{e}}^{- t (\mathcal{A} + \lambda_0) } -  1) f_2 \Vert_{X}\\
 & \leq C (\mathcal{T}_*) \Vert (1 -  {\rm{e}}^{- t \lambda_0} ) f_2 \Vert_{X} + C(\mathcal{T}_*)  t^q \Vert (\mathcal{A} + \lambda_0)^q f_2 \Vert_{X}\\
& \leq C(q, \mathcal{T}_*, \lambda_0) t^q \Vert (\mathcal{A} + \lambda_0)^q f_2 \Vert_{X}.
\end{align*}
As a result, we see that for each $\mathcal{T}_* >0$ and $0 < q \leq 1$ there is $ C(\mathcal{T}_*)  = C (q,\mathcal{T}_*, \lambda_0) >0$ such that for all $f_1 \in X$, $f_2 \in D ( (\mathcal{A} + \lambda_0)^q)$, and $0 < t \leq \mathcal{T}_*$ 
\begin{align}
\Vert (\mathcal{A} + \lambda_0 )^q {\rm{e}}^{- t \mathcal{A}} f_1 \Vert_{X} & \leq C(\mathcal{T}_*)  t^{-q} \Vert f_1 \Vert_{X},\label{eq87}\\
\Vert \mathcal{A} {\rm{e}}^{- t \mathcal{A}} f_1 \Vert_{X} & \leq C (\mathcal{T}_*) t^{-1} \Vert f_1 \Vert_{X},  \label{eq88}\\
\Vert ({\rm{e}}^{- t \mathcal{A}} -  1) f_2 \Vert_{X} & \leq  C(\mathcal{T}_*) t^q \Vert (\mathcal{A} + \lambda_0)^q f_2 \Vert_{X}.\label{eq89}
\end{align}

Now we derive \eqref{eq82}. Fix $V_0 \in X$ and $F \in L^p(0,\mathcal{T};X) \cap C^{\eta}_{loc}((0,\mathcal{T}); X)$ for some $1 \leq p <\infty$, $0< \eta \leq 1$, and $\mathcal{T} \in (0,\infty]$. Fix $\varepsilon, \mathcal{T}_* >0$ such that $\varepsilon <\mathcal{T}_* < \mathcal{T}$. By \eqref{eq84} and the H\"{o}lder inequality, we see that for $t \leq \mathcal{T}_*$
\begin{align*}
\Vert V (t) \Vert_X & \leq \Vert {\rm{e}}^{- t \mathcal{A}} V_0 \Vert_X + \int_0^t \Vert {\rm{e}}^{- (t-\tau) \mathcal{A}} F (\tau )\Vert_X { \ }d\tau\\
& \leq C(\lambda_0, \mathcal{T}_*) \Vert V_0 \Vert_X + C (\lambda_0, \mathcal{T}_*,p) \Vert F \Vert_{L^p(0,\mathcal{T};X)}< + \infty.
\end{align*}
Let $s,t >0$ such that $\varepsilon \leq s \leq t \leq \mathcal{T}_*$. Since
\begin{align*}
& V (t) = {\rm{e}}^{- t \mathcal{A}} V_0 + \int_0^{t} {\rm{e}}^{- ( t - \tau ) \mathcal{A} } F (\tau ) { \ }d \tau,\\
& V (s) = {\rm{e}}^{- s \mathcal{A}} V_0 + \int_0^{s} {\rm{e}}^{- ( s - \tau ) \mathcal{A} } F (\tau ) { \ }d \tau,
\end{align*}
we have
\begin{multline*}
V (t) - V (s) = \{ {\rm{e}}^{- ( t - s ) \mathcal{A}} - 1 \} {\rm{e}}^{- s \mathcal{A}}V_0\\
+ \int_{s}^{t} {\rm{e}}^{- ( t - \tau ) \mathcal{A}} F ( \tau ) { \ }d \tau + \int_0^{s} \{ {\rm{e}}^{- ( t - s ) \mathcal{A}} - 1 \}{\rm{e}}^{-( s - \tau ) \mathcal{A}} F (\tau ) { \ } d \tau.
\end{multline*}
Write $\tilde{\mathcal{A}} = \mathcal{A} + \lambda_0$. It is easy to check that
\begin{equation*}
\tilde{\mathcal{A}} V (t) - \tilde{\mathcal{A}} V (s) = \sum_{m=1}^7 H_m (t,s).
\end{equation*}
Here
\begin{align*}
H_1 = H_1 (s,t) &:= ( {\rm{e}}^{- (t -s ) \mathcal{A} } - 1 ) \tilde{\mathcal{A}} {\rm{e}}^{- s \mathcal{A}} V_0,\\
H_2 = H_2 (s , t ) &:= \tilde{\mathcal{A}} \int_{s}^{t} {\rm{e}}^{- (t - \tau ) \mathcal{A}}\{ F (\tau ) - F (t) \} { \ }d \tau,\\
H_3 = H_3 (s , t ) &:= \tilde{\mathcal{A}} \int_0^{s} \{ {\rm{e}}^{- ( t - s ) \mathcal{A}} - 1 \} {\rm{e}}^{- ( s - \tau ) \mathcal{A}} \{ F (\tau ) - F (s) \} { \ } d \tau ,\\
H_4 = H_4 (s , t) &:= \mathcal{A} \int_{s}^{t} {\rm{e}}^{- (t - \tau ) \mathcal{A}} F (t) { \ }d \tau,\\
H_5 = H_5 (s , t) &:= \mathcal{A} \int_0^{s} {\rm{e}}^{- (s - \tau ) \mathcal{A}} \{ {\rm{e}}^{- (t - s ) \mathcal{A}} - 1  \} F (s) { \ }d \tau,\\
H_6 = H_6 (s , t) &:= \lambda_0 \int_{s}^{t} {\rm{e}}^{- (t - \tau ) \mathcal{A}} F (t) { \ }d \tau,\\
H_7 = H_7 (s , t) &:= \lambda_0 \int_0^{s} \{ {\rm{e}}^{- (t - s ) \mathcal{A}} - 1  \} {\rm{e}}^{- (s - \tau ) \mathcal{A}} F (s) { \ }d \tau.
\end{align*}
We first study $\Vert H_1 (s,t) \Vert_X$ and $\Vert H_2(s,t) \Vert_X$. By \eqref{eq89}, we see that
\begin{align*}
\Vert H_1 \Vert_{X} & \leq C(\mathcal{T}_*) (t - s )^\eta \Vert \tilde{\mathcal{A}}^{1 + \eta} {\rm{e}}^{- s \mathcal{A}} V_0 \Vert_{X}\\
& \leq \frac{  C(\mathcal{T}_*)(t - s)^{\eta}}{\varepsilon^{1+\eta}} \Vert V_0 \Vert_{X}.
\end{align*}
Since $F \in C^{\eta}_{loc}((0,\mathcal{T}); X)$, we apply \eqref{eq87} to check that
\begin{equation*}
\Vert H_2 \Vert_{X} \leq \int_{s}^{t}  C(\mathcal{T}_*) \frac{(t - \tau )^\eta}{(t - \tau )} { \ }d \tau \leq C(\mathcal{T}_*) (t - s)^{\eta}.
\end{equation*}
Next, we consider $\Vert H_3(s,t) \Vert_X$. Since
\begin{multline*}
\tilde{\mathcal{A}} {\rm{e}}^{- (t-\tau ) \mathcal{A}} - \tilde{\mathcal{A}} {\rm{e}}^{ -(s -\tau ) \mathcal{A}}\\
 = \mathcal{A} {\rm{e}}^{- (t-\tau ) \mathcal{A}} - \mathcal{A} {\rm{e}}^{ -(s -\tau ) \mathcal{A}} + \lambda_0 {\rm{e}}^{- (t-\tau ) \mathcal{A}} - \lambda_0 {\rm{e}}^{ -(s -\tau ) \mathcal{A}}\\
  = \int_s^t \frac{d}{d \ell} \mathcal{A} {\rm{e}}^{- (\ell - \tau) \mathcal{A}} { \ }d \ell + \lambda_0 ( {\rm{e}}^{- ( t - s)\mathcal{A}}  - 1 ) {\rm{e}}^{- (s - \tau ) \mathcal{A}}\\
 = \int_s^t \mathcal{A}^2 {\rm{e}}^{- (\ell - \tau) \mathcal{A}} { \ }d \ell + \lambda_0 ( {\rm{e}}^{- ( t - s)\mathcal{A}}  - 1 ) {\rm{e}}^{- (s - \tau ) \mathcal{A}},
\end{multline*}
we use \eqref{eq88} and \eqref{eq89} to find that for all $f_* \in X$
\begin{multline}\label{eq8010}
\Vert \tilde{\mathcal{A}} {\rm{e}}^{- (t-\tau ) \mathcal{A}} f_* - \tilde{\mathcal{A}} {\rm{e}}^{ -(s -\tau ) \mathcal{A}} f_* \Vert_X\\
 \leq \int_s^t \Vert \mathcal{A}^2 {\rm{e}}^{- (\ell - \tau) \mathcal{A}} f_* \Vert_X{ \ }d \ell  + \lambda_0 \Vert ( {\rm{e}}^{- ( t - s)\mathcal{A}}  - 1 ) {\rm{e}}^{- (s - \tau ) \mathcal{A}} f_* \Vert_X\\
 \leq \int_s^t \frac{C(\mathcal{T}_*)}{(\ell - \tau)^2} { \ }d \ell \Vert f_* \Vert_X +  C (\mathcal{T}_*) (t-s) \Vert \tilde{A} {\rm{e}}^{ - (s- \tau ) \mathcal{A}} f_*  \Vert_X\\
 \leq \frac{ C(\mathcal{T}_*)(t-s) }{( t - \tau ) (s- \tau)} \Vert f_* \Vert_X + \frac{C(\mathcal{T}_*) (t-s) }{(s - \tau )} \Vert f_* \Vert_X.
\end{multline}
Applying \eqref{eq8010}, we see that
\begin{multline*}
\Vert H_3(s,t) \Vert_X \leq \int_0^s \Vert ( \tilde{\mathcal{A}} {\rm{e}}^{- (t-\tau ) \mathcal{A}} - \tilde{\mathcal{A}} {\rm{e}}^{ -(s -\tau ) \mathcal{A}}) \{ F (\tau) - F (s) \} \Vert_X { \ }d\tau\\
\leq \int_0^s \frac{ C(\mathcal{T}_*)(t-s) }{( t - \tau ) (s- \tau)} ( s - \tau)^\eta { \ }d \tau + \int_0^s \frac{C(\mathcal{T}_*) (t-s) }{(s - \tau )} ( s - \tau)^\eta { \ }d\tau\\
\leq  C (\mathcal{T}_*) (t - \tau )^\eta  + C(\mathcal{T}_*) (t- s)^\eta \leq C(\mathcal{T}_*) (t-s)^\eta.
\end{multline*}
Thirdly, we calculate $\Vert H_4(s,t) + H_5(s,t) \Vert_X$. From
\begin{equation*}
\frac{d}{d \tau } {\rm{e}}^{- \tau \mathcal{A}} \phi = - \mathcal{A} {\rm{e}}^{- \tau \mathcal{A}} \phi { \ \ \ }( \phi \in X),
\end{equation*}
we find that
\begin{equation*}
H_4 + H_5 = F (t) - F (s) - {\rm{e}}^{- (t - s) \mathcal{A}} \{ F(t) - F (s) \} - ({\rm{e}}^{- (t - s) \mathcal{A}} - 1){\rm{e}}^{- s \mathcal{A}} F (s).
\end{equation*}
By \eqref{eq85}, \eqref{eq87}, and \eqref{eq89}, we see that
\begin{align*}
\Vert H_4 + H_5 \Vert_{X} & \leq  C(\mathcal{T}_*) (t - s)^{\eta} +  C(\mathcal{T}_*) (t - s)^{\eta} \Vert \tilde{\mathcal{A}}^{\eta} {\rm{e}}^{- s \mathcal{A}} F (s) \Vert_{X}\\
& \leq   C(\mathcal{T}_*) (t - s)^{\eta} + \frac{ C(\mathcal{T}_*) (t - s)^{\eta}}{ \varepsilon^\eta} \sup_{\varepsilon \leq s \leq \mathcal{T}_*} \Vert F ( s ) \Vert_{X}.   
\end{align*}
Finally, we consider $\Vert H_6(s,t) \Vert_X$ and $\Vert H_7(s,t) \Vert_X$. From
From \eqref{eq84}, we check that
\begin{equation*}
\Vert H_6 (s,t) \Vert_X \leq C(\mathcal{T}_*) \sup_{\varepsilon \leq t \leq \mathcal{T}_*} \{ \Vert F (t) \Vert_X \} (t-s). 
\end{equation*}
Using \eqref{eq89} and \eqref{eq87}, we observe that
\begin{align*}
\Vert H_7 (s,t) \Vert_X & \leq  C(\mathcal{T}_*) \int_0^s (t-s)^{\eta} \Vert \tilde{\mathcal{A}}^{\eta} {\rm{e}}^{-(s- \tau)\mathcal{A}} F (s) \Vert_X { \ }d\tau \\
& \leq  C(\mathcal{T}_*) (t-s)^{\eta}  \int_0^s \frac{\Vert F (s) \Vert_X }{ (s - \tau )^{\eta} } d \tau\\
& \leq  C(\mathcal{T}_*) \mathcal{T}_*^{1-\eta} \sup_{\varepsilon \leq s \leq \mathcal{T}_* }\{ \Vert F (s) \Vert_X \} (t-s)^{\eta}.
\end{align*}
As a result, we have
\begin{equation*}
\Vert \tilde{\mathcal{A}} V (t) - \tilde{\mathcal{A}} V (s) \Vert_{X} \leq C(\lambda_0, \eta, \varepsilon , \mathcal{T}_*,\sup_{\varepsilon \leq \tau \leq \mathcal{T}_* }\{ \Vert F (\tau) \Vert_X \} ) (t - s )^{ \eta }. 
\end{equation*}
Therefore, we see that for each fixed $\varepsilon , \mathcal{T}_* >0$ such that $0< \varepsilon < \mathcal{T}_* <\mathcal{T}$
\begin{equation*}
\tilde{\mathcal{A}} V \in C^{ \eta }([\varepsilon , \mathcal{T}_*]; X).
\end{equation*}
This shows that
\begin{equation*}
\tilde{\mathcal{A}} V \in C_{loc}^{ \eta }( (0, \mathcal{T}); X).
\end{equation*}
From \eqref{eq83}, we find that
\begin{equation*}
V, \mathcal{A} V \in C_{loc}^{\eta }( (0, \mathcal{T}); X).
\end{equation*}
Since $F \in C_{loc}^{\eta} ((0,\mathcal{T});X)$ and $dV/{dt} = F - \mathcal{A} V$, we see that
\begin{equation*}
dV/{dt} \in C_{loc}^{ \eta }( (0, \mathcal{T}); X).
\end{equation*}
Therefore, we see \eqref{eq82}.
  \end{proof}

\end{document}